\documentclass{amsart}

\usepackage[margin=1.2in]{geometry}
\usepackage{amssymb}
\usepackage{mathrsfs}
\usepackage{thmtools}
\usepackage{graphicx}
\usepackage{pbox}

\usepackage{csquotes}

\usepackage{caption}

\usepackage{cite}

\newtheorem{theorem}{Theorem}[section]
\newtheorem{lemma}[theorem]{Lemma}
\newtheorem{conjecture}{Conjecture} 

\theoremstyle{definition}
\newtheorem{definition}[theorem]{Definition}

\newtheorem{corollary}[theorem]{Corollary}

\theoremstyle{remark}
\newtheorem{remark}[theorem]{Remark}

\begin{document}

%
%

\title{Genera of knots in the complex projective plane}

\author{Jake Pichelmeyer}

%
%

\begin{abstract}
Our goal is to systematically compute the $\mathbb{C}P^2$-genus of all prime knots up to 8-crossings. We obtain upper bounds on the $\mathbb{C}P^2$-genus via coherent band surgery. We obtain lower bounds by obstructing homological degrees of potential slice discs. The obstructions are pulled from a variety of sources in low-dimensional topology and adapted to $\mathbb{C}P^2$. There are 27 prime knots and distinct mirrors up to 7-crossings. We now know the $\mathbb{C}P^2$-genus of all but 2 of these knots. There are 64 prime knots and distinct mirrors up to 8-crossings. We now know the $\mathbb{C}P^2$-genus of all but 9 of these knots. Where the $\mathbb{C}P^2$-genus was not determined explicitly, it was narrowed down to 2 possibilities. As a consequence of this work, we show an  infinite family of knots such that the $\mathbb{C}P^2$-genus of each knot differs from that of it's mirror.
\end{abstract}

\maketitle

%
%

\section{Introduction}

Throughout this paper we work in the smooth category. All manifolds are considered to be connected, orientable, oriented, and compact unless otherwise stated. $D^n$ stands for a $n$-disc with boundary $S^{n-1}$, while $B^n$ stands for an open $n$-ball with no boundary. We use general notation and orientation conventions that are consistent with Livingston and Naik's excellent text on knot concordance \cite{livnai}. In particular, if $K$ is a knot in $S^3$, then $mK$ stands for it's mirror, which is the same knot in $S^3$ but with all positive crossings changed to negative crossings and vice versa. The same holds for links. We use notation for specific knots that aligns with Knotinfo \cite{knotinfo}. In particular, when one clicks on a knot $K$ at the Knotinfo site, we take the diagram shown on the left to be the knot $K$ and the diagram shown on the right to be the knot $mK$. For links we use the notation of Linkinfo \cite{linkinfo}. The top left corner diagram is $L$, while the bottom right diagram is $mL$.

\subsection{Background} Invariants of knots derived from surface genera have a long history dating back to 1935. Seifert had shown that for any knot $K$, one can algorithmically construct an orientable surface $S$ embedded in $S^3$ with boundary $K$ (a so-called \textit{Seifert surface}). Once it was known that every knot $K$ bounds a surface $S_K$ in $S^3$, it was natural to ask what the minimal genus of such a surface could be. The\textit{3-genus} (or \textit{Seifert genus}) $g_3(K)$ of a knot $K$ is defined to be just that \cite{Seifert1935}. By 1966, Fox and Milnor had extended this into 4 dimensions, giving rise to the \textit{smooth 4-genus} (or \textit{slice genus}). The smooth 4-genus $g_4(K)$ of a knot $K$ is defined to be the least genus among all orientable surfaces smoothly and properly embedded in $D^4$ with boundary $K$ \cite{Fox1962a} \cite{Fox1962b} \cite{Fox1966}. Since 1966, knot invariants involving the word 'genus' have proliferated. In the 3-dimensional world there is the \textit{3-genus} and the \textit{non-orientable 3-genus}. In the 4-dimensional world there is the \textit{smooth 4-genus}, the \textit{topological 4-genus}, the \textit{nonorientable smooth 4-genus}, the \textit{nonorientable topological 4-genus}, the \textit{Turaev genus},  the \textit{smooth concordance genus}, and the \textit{topological concordance genus} just to name a few. The unifying theme among all these knot invariants is that for a given knot $K$ they are each the minimal genus or first Betti number among a family of surfaces associated to $K$ with property set $P$. The manner in which surfaces are associated to $K$ and the property set $P$ together define the invariant uniquely. We will call the family of knot invariants which fit this definition \textit{genus knot invariants}.

The subject of this paper is to continue in some sense the most fundamental line of work related to genus knot invariants. The most natural extension of the original genus knot invariant $g_3$ is $g_4$ and the most natural extension of $g_4$ is the $M$-genus $g_M$, where $M$ ranges over all smooth closed 4-manifolds.

\begin{definition}
Let $K$ be a knot and $M$ a smooth closed 4-manifold. The \textit{$M$-genus of $K$}, denoted $g_M(K)$, is the least genus among all orientable surfaces $S_K$ embedded smoothly and properly in $M\setminus B^4$ with $\partial S_K=K$. If $g_M(K)=0$, we say that $K$ is \textit{slice} in $M$. If $D_K$ is a 2-disc smoothly and properly embedded in $M\setminus B^4$ with $\partial D_K = K$, then we say that $D_K$ is a \textit{slice disc}.
\end{definition}

\begin{remark}\label{remark 4-genus an upper bound on M-genus}
Observe that $g_4=g_{S^4}$ since $S^4\setminus B^4\cong D^4$.
\end{remark}

The simplest smooth closed 4-manifolds are $S^4, S^2\times S^2$, and $\mathbb{C}P^2$. The $S^4$-genus has been studied extensively since it's introduction in 1966. As of this writing, the $S^4$-genus is currently known for all but 27 of the 2,977 prime knots up to 12-crossings \cite{Lewark2017} \cite{knotinfo}. The $S^2\times S^2$-genus was completely determined in 1969 when Suzuki showed that all knots are slice in $S^2\times S^2$ \cite{Suzuki1969}. Much less is known about the $\mathbb{C}P^2$-genus than either the $S^2\times S^2$- or the $S^4$-genus. Hence the impetus for this work.

There are two general approaches to computing values for a given genus knot invariant. The first approach is to compute the invariant for all prime knots up to a certain crossing number. One starts with a low crossing number, computes the invariant for all prime knots up to that crossing number, and then works upward to ever higher numbers of crossings. After computing the invariant for all prime knots up to 5-crossings, for instance, one would then compute the invariant for all prime knots up to 6-crossings, and so on. The second approach is to compute the invariant for certain infinite families of knots, with torus knots being the most common choice. We will focus on the first approach in this paper, though we acquire results fitting the second approach as a consequence (see Section \ref{section computation}). In particular, it is our goal to compute the $\mathbb{C}P^2$-genus for all prime knots through 8-crossings.

\subsection{What was known} 

The first notable work involving the $\mathbb{C}P^2$-genus  was presented by Yasuhara in 1991 and 1992 when he showed that for $x\geq 2$, the torus knot $T(2,2x+1)$ is not slice in $\mathbb{C}P^2$ \cite{Yasuhara1991} \cite{Yasuhara1992}. In 2009, Ait Nouh computed the explicit $\mathbb{C}P^2$-genus for a finite set of torus knots \cite{AitNouh2009}. Namely, he showed that
\begin{align*}
g_{\mathbb{C}P^2}(T(2,2x+1))		
&= x-1 \hspace{0.25cm}\text{when}\hspace{0.25cm} 1\leq x \leq 8 \\
g_{\mathbb{C}P^2}(T(-2,2x+1))		
&= 0 \hspace{0.25cm}\text{when}\hspace{0.25cm} 1\leq x\leq 4 \\
g_{\mathbb{C}P^2}(T(-2,2x+1))		
&= 1 \hspace{0.25cm}\text{when}\hspace{0.25cm} x=5.
\end{align*}

\begin{remark}
Observe that the $\mathbb{C}P^2$-genus of a knot $K$ may differ from it's mirror $mK$. The knot $T(2,5)$ has $\mathbb{C}P^2$-genus 1 while it's mirror $T(-2,5)$ has $\mathbb{C}P^2$-genus 0. We will see in Section \ref{section computation} that there are infinitely many pairs $K,mK$ with differing $\mathbb{C}P^2$-genus.
\end{remark}

\begin{remark}
It only makes our goal more difficult to accomplish that the $\mathbb{C}P^2$-genus may differ between a knot $K$ and it's mirror $mK$. At all crossing numbers, there are simply more knots to compute the $\mathbb{C}P^2$-genus for. However, all it not lost. As we will see with Theorem \ref{mainresult2} and Corollary \ref{corollary lawson}, knowledge about the $\mathbb{C}P^2$-genus of a knot sometimes translates into knowledge about the $\mathbb{C}P^2$-genus of it's mirror.
\end{remark}

There are 64 knots and distinct mirrors up to 8-crossings. The work by Ait Nouh gives explicit computations for 4 of these knots, namely
\[5_1,\hspace{.2cm} m5_1,\hspace{.2cm} 7_1,\hspace{.2cm} m7_1.\]

Yasuhara's Lemma 1.9 \cite{Yasuhara1992}, which he attributes to Weintraub, shows that any knot with unknotting number 1 is slice in $\mathbb{C}P^2$. For the set of knots we are considering, this includes
\[3_1,\hspace{.2cm} m3_1,\hspace{.2cm} 4_1,\hspace{.2cm} 5_2,\hspace{.2cm} m5_2,\hspace{.2cm} 6_1,\hspace{.2cm} m6_1,\hspace{.2cm} 6_2,\hspace{.2cm} m6_2,\hspace{.2cm} 6_3,\hspace{.2cm} 7_2,\hspace{.2cm} m7_2,\hspace{.2cm} 7_6,\hspace{.2cm} m7_6,\hspace{.2cm} 7_7,\hspace{.2cm} m7_7\]
\[8_1,\hspace{.2cm} m8_1,\hspace{.2cm} m8_7,\hspace{.2cm} 8_7,\hspace{.2cm} 8_9,\hspace{.2cm} 8_{11},\hspace{.2cm} m8_{11},\hspace{.2cm} 8_{13},\hspace{.2cm} m8_{13},\hspace{.2cm} 8_{14},\hspace{.2cm} m8_{14},\hspace{.2cm} 8_{17},\hspace{.2cm} 8_{20},\hspace{.2cm} m8_{20},\hspace{.2cm} 8_{21},\hspace{.2cm} m8_{21}.\]
We will give more details about how the unknotting number provides an upper bound on the $\mathbb{C}P^2$-genus in Section \ref{section upper bounds}.

Per Lemma \ref{lemma 4genus upper bound}, since $g_{\mathbb{C}P^2}(K)\leq g_4(K)$, it follows that any knot which is slice in $S^4$ is slice in $\mathbb{C}P^2$. Prime knots up to 8-crossings which have 4-genus 0 but haven't already been listed include
\[0_1,\hspace{.2cm} 8_8,\hspace{.2cm} m8_8.\]

\subsection{What we've shown}

A definition of \textit{coherent band surgery} may be found both in \cite{Moore2018} and in Section \ref{section upper bounds} of this paper. For those readers who consult \cite{Yasuhara1992}, Yasuhara defines \textit{m-fusion} and \textit{m-fission}. When $m=1$, $m$-fusion and $m$-fission are coherent band surgeries.

Let $S_K$ be a properly embedded surface in $\mathbb{C}P^2\setminus B^4$ with $\partial S_K = K\subset \partial (\mathbb{C}P^2\setminus B^4)$. We know that $S_K$ represents some class $[S_K]\in H_2(\mathbb{C}P^2\setminus B^4, \partial; \mathbb{Z})$. Let $\gamma:=[\mathbb{C}P^1]$ denote the generator of $H_2(\mathbb{C}P^2\setminus B^4,\partial;\mathbb{Z})\cong \mathbb{Z}$. Then $[S_K]=d\gamma$. We call $d$ the \textit{degree} of $S_K$. If $S_K$ is a 2-disc, then we say that $d$ is a \textit{slice degree} of $K$.

\begin{remark}
Given a knot $K$, one can form a new knot $rK$ by reversing the string-orientation of $K$ \cite{livnai}. If $K$ bounds a surface $S_K$ that is smoothly and properly embedded in $\mathbb{C}P^2\setminus B^4$ with degree $d$, then $rK$ bounds a surface $\overline{S_K}$ smoothly and properly embedded in $\mathbb{C}P^2\setminus B^4$  with degree $-d$. Since all our computations will involve squaring the degree of such a surface $S_K$, the string orientation will not be relevant in general. The only time we will pay any attention to the string-orientation of a knot is to ensure that the band surgeries we perform are coherent and as a rare technical detail.
\end{remark}

We've proven 3 main theorems during the course of this work.

\begin{theorem}\label{mainresult1}.
Let $K$ be a knot such that $mK$ is obtained from one of the links below via coherent band surgery. Then $K$ bounds a properly embedded disc $D_K$ in $\mathbb{C}P^2\setminus B^4$ with $[D_K]=d\gamma$. In particular, $K$ is slice in $\mathbb{C}P^2$.
\vspace{0.1cm}
\textup{
\begin{center}
\begin{tabular}{p{1cm} p{11cm}}
$|d|$ 	& Links \\
0		& $m$L2a1$\{1\}$,\; L5a1$\{0\}$,\; L7n2$\{0\}\{1\}$,\;
\\
1		& Unlink,\; L4a1$\{0\}$,\; $m$L7a4$\{0\}\{1\},$\; $m$L7n1$\{1\}$,\; $3_1\#m$L4a1$\{1\}$
\\
2		& L2a1$\{1\}$,\; $m$L5a1$\{0\}$\; $m$L7n2$\{0\}\{1\}$
\\
3		& L4a1$\{1\}$,\; L7a3$\{0\}\{1\}$,\; $m$L7n1$\{0\}$,\; $3_1\#m$L4a1$\{0\}$
\\
\end{tabular}
\end{center}
}
\end{theorem}

\begin{theorem}\label{mainresult2}
Let $K$ be an alternating knot with $|\sigma(K)|=4$. Then either $K$ or $mK$ fails to be slice in $\mathbb{C}P^2$.
\end{theorem}

\begin{theorem}\label{mainresult3}
Let $K$ be an alternating knot with $\sigma(K)=4$, $g_4(K)\leq 2$, and Arf$(K)=0$. Then $K$ is not slice in $\mathbb{C}P^2$.
\end{theorem}

Using coherent band surgery, the following prime knots and mirrors up to 8-crossings were found to satisfy the hypothesis of Theorem \ref{mainresult1}
\[m7_3,\hspace{.2cm} 7_4,\hspace{.2cm} 8_3,\hspace{.2cm} 8_4,\hspace{.2cm} m8_5,\hspace{.2cm} 8_6,\hspace{.2cm} m8_6,\hspace{.2cm} 8_{12},\hspace{.2cm} m8_{19}.\]
It follows from  Theorem \ref{mainresult1} that all the knots in the list directly above are slice in $\mathbb{C}P^2$. We found several other knots of 9- and 10-crossings to which Theorem \ref{mainresult1} applies, but since these are outside the scope of our main goal, we move their mention to Section \ref{section computation}. The coherent band surgeries for all knots to which we've found Theorem \ref{mainresult1} to apply can be found in Appendix \ref{appendix surgeries knots to links}.

Theorem \ref{mainresult2} is largely a companion to Theorem \ref{mainresult1}. For alternating knots $K$ with $\sigma(K)=\pm 4$, if we know that one of $K, mK$ is slice in $\mathbb{C}P^2$, then by Theorem \ref{mainresult2}, the other cannot be slice in $\mathbb{C}P^2$. Prime knots and mirrors up to 8-crossings which satisfy the hypothesis of Theorem \ref{mainresult2} and have a mirror that is slice in $\mathbb{C}P^2$ include
\[m7_3,\hspace{.2cm} 8_5.\]
As we will see in Section \ref{section upper bounds}, since the knots above both have unknotting number 2 their $\mathbb{C}P^2$-genus is 1. Using Theorem \ref{mainresult1} and Theorem \ref{mainresult2}, we were able to show that for an infinite family of knots $\{K_n\}$, $K_n$ and $mK_n$ have differing $\mathbb{C}P^2$-genus for each $n\in\mathbb{N}$. This is explained in more detail in Section \ref{section computation}.

There are 220 prime knot up to 12-crossings that satisfy the hypothesis of Theorem \ref{mainresult3}. These can be found easily by using the search function on Knotinfo \cite{knotinfo}. However, we are concerning ourselves primarily with those up to 8-crossings. Such knots include
\[m7_5,\hspace{.2cm} m8_2,\hspace{.2cm} m8_{15}\]
Due to their unknotting number being 2, we can explicitly compute the $\mathbb{C}P^2$-genus of these knots to be 1. More details about how the unknotting number gives an upper bound on the $\mathbb{C}P^2$-genus can be found in Section \ref{section upper bounds}.

Using the definition of knot concordance, we were able to show that 
\[8_{10},\hspace{.2cm} m8_{10}\]
are both slice in $\mathbb{C}P^2$. This is explained in more detail in Section \ref{section upper bounds}. As with the results above, we were able to use knot concordance to compute the $\mathbb{C}P^2$-genus for more than just these knots, but for the sake of focus we place those results in Section \ref{section computation}.

A series of tables is provided below. Each table contains all the prime knots and distinct mirrors up to 8-crossings of a particular signature $\sigma$ and Arf invariant. For example, the first table lists all the prime knots and distinct mirrors up to 8-crossings with signature and Arf invariant 0. Each prime knot and distinct mirror up to 8-crossings is contained in one of the tables. When the $\mathbb{C}P^2$-genus is known explicitly it is given. When it is not known completely, the set of possibilites is given. In all cases where the $\mathbb{C}P^2$-genus is not known, there are exactly two possibilities. For each knot, the set of possible slice degrees as allowed by Corollaries \ref{corollary ozvath szabo}, \ref{corollary kronheimer mrowka}, \ref{corollary gilmer viro}, \ref{corollary lawson}, and \ref{corollary robertello} are listed. If the author has explicitly constructed a slice disc with a particular slice degree, then it is listed as a realized slice degree.

\begin{center}

\begin{tabular}{|p{1cm} |p{0.75cm} |p{0.5cm} |p{0.5cm} |p{0.75cm} |p{0.5cm} |p{0.5cm} |p{3cm} | p{3cm} | }
\hline $K$ & $g_{\mathbb{C}P^2}$ & alt? & $\sigma$ & Arf & $g_4$ & $u$ & \pbox{3cm}{\footnotesize Possible Slice Degrees}  &
\pbox{3cm}{\footnotesize Realized Slice Degrees}  \\
\hline
$0_1$ & $0$ & Y & $0$ & $0$ & $0$ & $0$ & 2, 0, 1 & 2, 0, 1 \\
\hline
$6_1$ & $0$ & Y & $0$ & $0$ & $0$ & $1$ & 2, 0, 1 & 2 \\
\hline
$m6_1$ & $0$ & Y & $0$ & $0$ & $0$ & $1$ & 2, 0, 1 & 0, 1 \\
\hline
$8_3$ & $0$ & Y & $0$ & $0$ & $1$ & $2$ & 2, 0, 1 & 1 \\
\hline
$8_8$ & $0$ & Y & $0$ & $0$ & $0$ & $2$ & 2, 0, 1 & 1 \\
\hline
$m8_8$ & $0$ & Y & $0$ & $0$ & $0$ & $2$ & 2, 0, 1 & 2 \\
\hline
$8_9$ & $0$ & Y & $0$ & $0$ & $0$ & $1$ & 2, 0, 1 & 1 \\
\hline
$8_{20}$ & $0$ & N & $0$ & $0$ & $0$ & $1$ & 2, 0, 1 & 1 \\
\hline
$m8_{20}$ & $0$ & N & $0$ & $0$ & $0$ & $1$ & 2, 0, 1 & 1 \\
\hline
\end{tabular}
\vspace{0.5cm}

\begin{tabular}{|p{1cm} |p{0.75cm} |p{0.5cm} |p{0.5cm} |p{0.75cm} |p{0.5cm} |p{0.5cm} |p{3cm} | p{3cm} | }
\hline $K$ & $g_{\mathbb{C}P^2}$ & alt? & $\sigma$ & Arf & $g_4$ & $u$ & \pbox{3cm}{\footnotesize Possible Slice Degrees}  &
\pbox{3cm}{\footnotesize Realized Slice Degrees}  \\
\hline
$4_1$ & $0$ & Y & $0$ & $1$ & $1$ & $1$ & 2, 0 & 2, 0 \\
\hline
$6_3$ & $0$ & Y & $0$ & $1$ & $1$ & $1$ & 2, 0 & 2, 0 \\
\hline
$7_7$ & $0$ & Y & $0$ & $1$ & $1$ & $1$ & 2, 0, 3 & 2 \\
\hline
$m7_7$ & $0$ & Y & $0$ & $1$ & $1$ & $1$ & 2, 0, 3 & 0 \\
\hline
$8_1$ & $0$ & Y & $0$ & $1$ & $1$ & $1$ & 2, 0, 3 & \\
\hline
$m8_1$ & $0$ & Y & $0$ & $1$ & $1$ & $1$ & 2, 0, 3 & 0 \\
\hline
$8_{12}$ & $0$ & Y & $0$ & $1$ & $1$ & $2$ & 2, 0 & 2, 0 \\
\hline
$8_{13}$ & $0$ & Y & $0$ & $1$ & $1$ & $1$ & 2, 0, 3 &  \\
\hline
$m8_{13}$ & $0$ & Y & $0$ & $1$ & $1$ & $1$ & 2, 0, 3 & \\
\hline
$8_{17}$ & $0$ & Y & $0$ & $1$ & $1$ & $1$ & 2, 0, 3 & 2 \\
\hline
$8_{18}$ & $\{0,1\}$ & Y & $0$ & $1$ & $1$ & $2$ & 2, 0 & \\
\hline
\end{tabular}
\vspace{0.5cm}

\begin{tabular}{|p{1cm} |p{0.75cm} |p{0.5cm} |p{0.5cm} |p{0.75cm} |p{0.5cm} |p{0.5cm} |p{3cm} | p{3cm} | }
\hline $K$ & $g_{\mathbb{C}P^2}$ & alt? & $\sigma$ & Arf & $g_4$ & $u$ & \pbox{3cm}{\footnotesize Possible Slice Degrees}  &
\pbox{3cm}{\footnotesize Realized Slice Degrees}  \\
\hline
$5_2$ & $0$ & Y & $-2$ & $0$ & $1$ & $1$ & 0, 1 & 0, 1 \\
\hline
$7_4$ & $0$ & Y & $-2$ & $0$ & $1$ & $2$ & 0, 1 & 0, 1 \\
\hline
$8_6$ & $0$ & Y & $-2$ & $0$ & $1$ & $2$ & 0, 1 & 0 \\
\hline
$m8_7$ & $0$ & Y & $-2$ & $0$ & $1$ & $1$ & 0, 1 & 0 \\
\hline
$8_{14}$ & $0$  & Y & $-2$ & $0$ & $1$ & $1$ & 0, 1 & 0 \\
\hline
$8_{21}$ & $0$ & N & $-2$ & $0$ & $1$ & $1$ & 0, 1 & 0 \\
\hline
\end{tabular}
\vspace{0.5cm}

\begin{tabular}{|p{1cm} |p{0.75cm} |p{0.5cm} |p{0.5cm} |p{0.75cm} |p{0.5cm} |p{0.5cm} |p{3cm} | p{3cm} | }
\hline $K$ & $g_{\mathbb{C}P^2}$ & alt? & $\sigma$ & Arf & $g_4$ & $u$ & \pbox{3cm}{\footnotesize Possible Slice Degrees}  &
\pbox{3cm}{\footnotesize Realized Slice Degrees}  \\
\hline
$m5_2$ & $0$ & Y & $2$ & $0$ & $1$ & $1$ & 2 & 2 \\
\hline
$m7_4$ & $\{0,1\}$ & Y & $2$ & $0$ & $1$ & $2$ & 2 &  \\
\hline
$m8_6$ & $0$ & Y & $2$ & $0$ & $1$ & $2$ & 2 & 2 \\
\hline
$8_7$ & $0$ & Y & $2$ & $0$ & $1$ & $1$ & 2 & 2 \\
\hline
$m8_{14}$ & $0$  & Y & $2$ & $0$ & $1$ & $1$ & 2 & 2 \\
\hline
$m8_{21}$ & $0$ & N & $2$ & $0$ & $1$ & $1$ & 2, 0, 1 & 2 \\
\hline
\end{tabular}
\vspace{0.5cm}

\begin{tabular}{|p{1cm} |p{0.75cm} |p{0.5cm} |p{0.5cm} |p{0.75cm} |p{0.5cm} |p{0.5cm} |p{3cm} | p{3cm} | }
\hline $K$ & $g_{\mathbb{C}P^2}$ & alt? & $\sigma$ & Arf & $g_4$ & $u$ & \pbox{3cm}{\footnotesize Possible Slice Degrees}  &
\pbox{3cm}{\footnotesize Realized Slice Degrees}  \\
\hline
$3_1$ & $0$ & Y & $-2$ & $1$ & $1$ & $1$ & 0 & 0 \\
\hline
$6_2$ & $0$ & Y & $-2$ & $1$ & $1$ & $1$ & 0, 3 & 0 \\
\hline
$7_2$ & $0$ & Y & $-2$ & $1$ & $1$ & $1$ & 0 & 0 \\
\hline
$7_6$ & $0$ & Y & $-2$ & $1$ & $1$ & $1$ & 0, 3 & 0 \\
\hline
$m8_4$ & $\{0,1\}$  & Y & $-2$ & $1$ & $1$ & $2$ & 0 & \\
\hline
$m8_{10}$ & $0$ & Y & $-2$ & $1$ & $1$ & $2$ & 0 & 0 \\
\hline
$8_{11}$ & $0$ & Y & $-2$ & $1$ & $1$ & $1$ & 0, 3 & 0 \\
\hline
$8_{16}$ & $\{0,1\}$ & Y & $-2$ & $1$ & $1$ & $2$ & 0, 3 & \\
\hline
\end{tabular}
\vspace{0.5cm}

\begin{tabular}{|p{1cm} |p{0.75cm} |p{0.5cm} |p{0.5cm} |p{0.75cm} |p{0.5cm} |p{0.5cm} |p{3cm} | p{3cm} | }
\hline $K$ & $g_{\mathbb{C}P^2}$ & alt? & $\sigma$ & Arf & $g_4$ & $u$ & \pbox{3cm}{\footnotesize Possible Slice Degrees}  &
\pbox{3cm}{\footnotesize Realized Slice Degrees}  \\
\hline
$m3_1$ & $0$ & Y & $2$ & $1$ & $1$ & $1$ & 2, 3 & 2, 3 \\
\hline
$m6_2$ & $0$ & Y & $2$ & $1$ & $1$ & $1$ & 2, 3 & 2 \\
\hline
$m7_2$ & $0$ & Y & $2$ & $1$ & $1$ & $1$ & 2, 3 & 2, 3 \\
\hline
$m7_6$ & $0$ & Y & $2$ & $1$ & $1$ & $1$ & 2, 3 & 2 \\
\hline
$8_4$ & $0$ & Y & $2$ & $1$ & $1$ & $2$ & 2, 3 & 3 \\
\hline
$8_{10}$ & $0$ & Y & $2$ & $1$ & $1$ & $2$ & 2, 3 & 2, 3 \\
\hline
$m8_{11}$ & $0$ & Y & $2$ & $1$ & $1$ & $1$ & 2, 3 & 2 \\
\hline
$m8_{16}$ & $\{0,1\}$ & Y & $2$ & $1$ & $1$ & $2$ & 2, 3 & \\
\hline
\end{tabular}
\vspace{0.5cm}

\begin{tabular}{|p{1cm} |p{0.75cm} |p{0.5cm} |p{0.5cm} |p{0.75cm} |p{0.5cm} |p{0.5cm} |p{3cm} | p{3cm} | }
\hline $K$ & $g_{\mathbb{C}P^2}$ & alt? & $\sigma$ & Arf & $g_4$ & $u$ & \pbox{3cm}{\footnotesize Possible Slice Degrees}  &
\pbox{3cm}{\footnotesize Realized Slice Degrees}  \\
\hline
$7_5$ & $\{0,1\}$ & Y & $-4$ & $0$ & $2$ & $2$ & 1 &  \\
\hline
$8_2$ & $\{0,1\}$ & Y & $-4$ & $0$ & $2$ & $2$ & 1 &  \\
\hline
$8_{15}$ & $\{0,1\}$ & Y & $-4$ & $0$ & $2$ & $2$ & 1 &  \\
\hline
\end{tabular}
\vspace{0.5cm}

\begin{tabular}{|p{1cm} |p{0.75cm} |p{0.5cm} |p{0.5cm} |p{0.75cm} |p{0.5cm} |p{0.5cm} |p{3cm} | p{3cm} | }
\hline $K$ & $g_{\mathbb{C}P^2}$ & alt? & $\sigma$ & Arf & $g_4$ & $u$ & \pbox{3cm}{\footnotesize Possible Slice Degrees}  &
\pbox{3cm}{\footnotesize Realized Slice Degrees}  \\
\hline
$m7_5$ & $1$ & Y & $4$ & $0$ & $2$ & $2$ & - & - \\
\hline
$m8_2$ & $1$ & Y & $4$ & $0$ & $2$ & $2$ & - & - \\
\hline
$m8_{15}$ & $1$ & Y & $4$ & $0$ & $2$ & $2$ & - & - \\
\hline
\end{tabular}
\vspace{0.5cm}

\begin{tabular}{|p{1cm} |p{0.75cm} |p{0.5cm} |p{0.5cm} |p{0.75cm} |p{0.5cm} |p{0.5cm} |p{3cm} | p{3cm} | }
\hline $K$ & $g_{\mathbb{C}P^2}$ & alt? & $\sigma$ & Arf & $g_4$ & $u$ & \pbox{3cm}{\footnotesize Possible Slice Degrees}  &
\pbox{3cm}{\footnotesize Realized Slice Degrees}  \\
\hline
$5_1$ & $1$ & Y & $-4$ & $1$ & $2$ & $2$ & - & - \\
\hline
$7_3$ & $1$ & Y & $-4$ & $1$ & $2$ & $2$ & - & - \\
\hline
$8_5$ & $1$ & Y & $-4$ & $1$ & $2$ & $2$ & - & - \\
\hline
\end{tabular}
\vspace{0.5cm}

\begin{tabular}{|p{1cm} |p{0.75cm} |p{0.5cm} |p{0.5cm} |p{0.75cm} |p{0.5cm} |p{0.5cm} |p{3cm} | p{3cm} | }
\hline $K$ & $g_{\mathbb{C}P^2}$ & alt? & $\sigma$ & Arf & $g_4$ & $u$ & \pbox{3cm}{\footnotesize Possible Slice Degrees}  &
\pbox{3cm}{\footnotesize Realized Slice Degrees}  \\
\hline
$m5_1$ & $0$ & Y & $4$ & $1$ & $2$ & $2$ & 3 & 3 \\
\hline
$m7_3$ & $0$ & Y & $4$ & $1$ & $2$ & $2$ & 3 & 3 \\
\hline
$m8_5$ & $0$ & Y & $4$ & $1$ & $2$ & $2$ & 3 & 3 \\
\hline
\end{tabular}
\vspace{0.5cm}

\begin{tabular}{|p{1cm} |p{0.75cm} |p{0.5cm} |p{0.5cm} |p{0.75cm} |p{0.5cm} |p{0.5cm} |p{3cm} | p{3cm} | }
\hline $K$ & $g_{\mathbb{C}P^2}$ & alt? & $\sigma$ & Arf & $g_4$ & $u$ & \pbox{3cm}{\footnotesize Possible Slice Degrees}  &
\pbox{3cm}{\footnotesize Realized Slice Degrees}  \\
\hline
$7_1$ & $2$ & Y & $-6$ & $0$ & $3$ & $3$ & - & - \\
\hline
\end{tabular}
\vspace{0.5cm}

\begin{tabular}{|p{1cm} |p{0.75cm} |p{0.5cm} |p{0.5cm} |p{0.75cm} |p{0.5cm} |p{0.5cm} |p{3cm} | p{3cm} | }
\hline $K$ & $g_{\mathbb{C}P^2}$ & alt? & $\sigma$ & Arf & $g_4$ & $u$ & \pbox{3cm}{\footnotesize Possible Slice Degrees}  &
\pbox{3cm}{\footnotesize Realized Slice Degrees}  \\
\hline
$m7_1$ & $0$ & Y & $6$ & $0$ & $3$ & $3$ & 4, 1 & 4 \\
\hline
\end{tabular}
\vspace{0.5cm}

\begin{tabular}{|p{1cm} |p{0.75cm} |p{0.5cm} |p{0.5cm} |p{0.75cm} |p{0.5cm} |p{0.5cm} |p{3cm} | p{3cm} | }
\hline $K$ & $g_{\mathbb{C}P^2}$ & alt? & $\sigma$ & Arf & $g_4$ & $u$ & \pbox{3cm}{\footnotesize Possible Slice Degrees}  &
\pbox{3cm}{\footnotesize Realized Slice Degrees}  \\
\hline
$8_{19}$ & $\{1,2\}$ & N & $-6$ & $1$ & $3$ & $3$ & - & - \\
\hline
\end{tabular}
\vspace{0.5cm}

\begin{tabular}{|p{1cm} |p{0.75cm} |p{0.5cm} |p{0.5cm} |p{0.75cm} |p{0.5cm} |p{0.5cm} |p{3cm} | p{3cm} | }
\hline $K$ & $g_{\mathbb{C}P^2}$ & alt? & $\sigma$ & Arf & $g_4$ & $u$ & \pbox{3cm}{\footnotesize Possible Slice Degrees}  &
\pbox{3cm}{\footnotesize Realized Slice Degrees}  \\
\hline
$m8_{19}$ & $0$ & N & $6$ & $1$ & $3$ & $3$ & 4, 3 & 3 \\
\hline
\end{tabular}
\vspace{0.5cm}

\end{center}

\subsection{Where to go}

Note that we have not found a case where two knots have the same signature and Arf invariant yet do not have the same $\mathbb{C}P^2$-genus. This leads to an obvious conjecture:

\begin{conjecture}\label{conjecture same sig arf same cp2 genus}
Let $K_1, K_2$ be knots with the same signature and Arf invariant. Then $K_1, K_2$ have the same $\mathbb{C}P^2$-genus.
\end{conjecture}

For those remaining 9 prime knots of 7- and 8-crossings for which the $\mathbb{C}P^2$-genus is not definitively known, we reduced the set of possibilities down to a 2-element set: either $\{0,1\}$ or $\{1,2\}$. Following Conjecture \ref{conjecture same sig arf same cp2 genus}, the author suspects that $m7_4, m8_4, 8_{16}, m8_{16}$, and $8_{18}$ have $\mathbb{C}P^2$-genus 0, while $7_5$, $8_2$, and $8_{15}$ have $\mathbb{C}P^2$-genus 1, and $8_{19}$ has $\mathbb{C}P^2$-genus 2.

In the case where we have not fully obstructed a knot from being slice, we have narrowed down the possible slice degrees to at most two. It seems reasonable that one could obstruct some remaining slice degrees using techniques such as the popular Donaldson diagonalization argument \cite{Lisca2007} \cite{Williams2008} \cite{Jabuka2018} adapted to $\mathbb{C}P^2\setminus B^4$. This will be the author's next approach.

\subsection{How this paper is structured}

In Section \ref{section upper bounds} we give more details as to how the smooth 4-genus, unknotting number, and known knot concordances provide us with upper bounds. We follow by proving Theorem \ref{mainresult1}. In Section \ref{section lower bounds} we provide several utility corollaries. Each allows the obstruction of a certain subset of slice degrees. Using these, we prove Theorems \ref{mainresult2} and \ref{mainresult3}. In Section \ref{section computation} we explain our computations of the $\mathbb{C}P^2$-genus for a finite set of prime knots of $9-$ and $10-$crossings and an infinite family of knots that were outside the primary scope of this work. In Appendix \ref{appendix surgeries for MR1} we show the coherent band surgeries required to fully prove Theorem \ref{mainresult1}. In Appendix \ref{appendix surgeries knots to links} we show the coherent band surgeries required to justify our applications of Theorem \ref{mainresult1}.

\subsection{Acknowledgements}

The author would like to thank Rustam Sadykov and Dave Auckly for helpful discussions, Victor Turchin and Mark Hughes for being thoughtful listeners, and Akira Yasuhara and Chuck Livingston for helpful comments on an early draft of this paper.

\section{Upper bounds}\label{section upper bounds}

There are two preliminaries which the reader may find helpful in order to better understand both the sliceness conditions in Subsection \ref{subsection trivial sliceness} and the proof of Theorem \ref{mainresult1} in Subsection \ref{subsection proof of mainresult1}.

The first preliminary is a discussion of coherent band surgery and the cobordisms related to this operation. This is covered in Subsection \ref{subsection band surgery}. The second preliminary involves a discussion about the handlebody decomposition of $\mathbb{C}P^2\setminus B^4$ and how its elements play a role in the constructive aspects of this work. This is discussed in Subsection \ref{subsection elements of CP2lessB4}.

\subsection{Band surgery}\label{subsection band surgery}

\begin{definition}
Let $L$ be a link in $S^3$. A \textit{band surgery} on $L$ is an embedding $b:I\times I\rightarrow S^3$ such that $L\cap b(I\times I)=b(I\times \partial I)$. The surgery is \textit{coherent} if the link
\[L'=\big[ L\setminus \text{int}[b(I\times\partial I)]\big] \cup_{\partial I\times \partial I}  b(\partial I \times I) \]
is oriented. Otherwise, the surgery is \textit{non-coherent}. We say that the link $L'$ is obtained from $L$ via the band surgery $b$.
\end{definition}

\begin{remark}
A band surgery may also be referred to as a \textit{band move} in the literature.
\end{remark}

\begin{figure}[h]
\includegraphics[width=.8\textwidth]{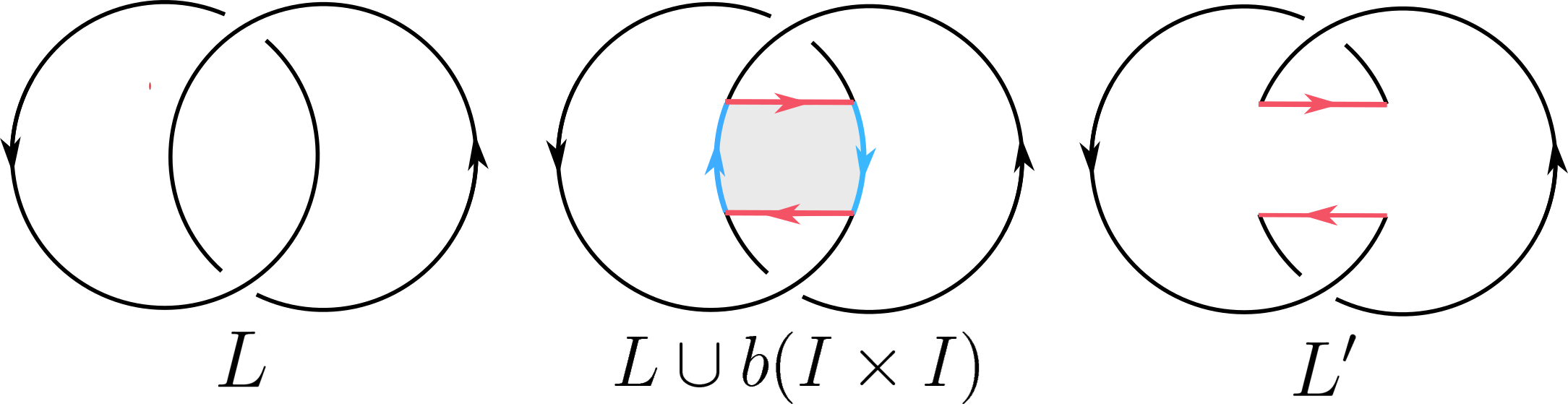}
\caption{A coherent band surgery from link $L$ to link $L'$.}\label{figure band surgery}
\end{figure}

Let $L'$ be a link obtained from $L$ via band surgery $b$. If $b(I\times I)$ intersects only a single component of $L$, then $L'$ will have exactly one more component than $L$. If $b(I\times I)$ intersects two components of $L$, then $L'$ will have exactly one less component than $L$. A band surgery either joins two components of $L$ into one or splits one component of $L$ into two. All components of $L$ that do not intersect $b(I\times I)$ remain unchanged. It follows from elementary surgery theory that there is a genus 0 cobordism $C\subset S^3\times I$ between $L$ and $L'$ where $C$ is a "pair of pants" between the component(s) that are changed and a series of disjoint cylinders elsewise.

\begin{figure}[h]
\includegraphics[width=.7\textwidth]{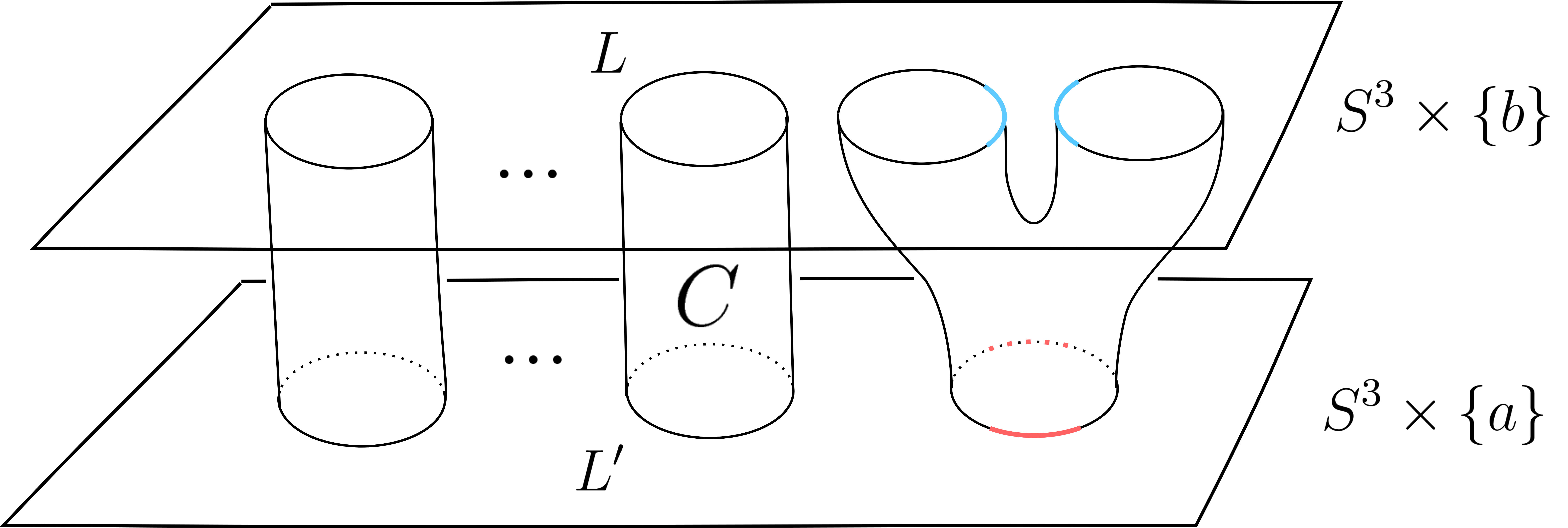}
\caption{A cartoon of a genus 0 cobordism $C$ in $S^3\times I$ with boundary components $L\subset S^3\times\{b\}$ and $L'\subset  S^3\times\{a\}$, where $0\leq a<b\leq 1$. Such a corbordism would be the result of performing a coherent band surgery to obtain $L'$ from $L$.}\label{figure band surgery cobordism}
\end{figure}

\subsection{Handlebody decomposition of $\mathbb{C}P^2\setminus B^4$}\label{subsection elements of CP2lessB4}

\subsubsection{Basic decomposition}

The handlebody decomposition of $\mathbb{C}P^2\setminus B^4$ is a single 4-dimensional 0-handle $h^0\cong D^4$ and a single 4-dimensional 0-handle $h^2\cong D^2\times D^2$ together with a gluing  map $\phi:\partial D^2 \times D^2 \rightarrow \partial h^0$. Thinking of $S^1\times D^2$ as a trivial $D^2$-fiber bundle over $S^1$, $\phi$ maps the 0-section of $S^1\times D^2$ to an unknot in $\partial h^0$, while mapping the $D^2$ fibers of $S^1\times D^2$ into $\partial h^0$ so that the fibers have exactly one full positive \enquote{twist.} Details of this construction can be found in Scorpan's illustrative text on 4-manifolds \cite{Scorpan2005}.

\subsubsection{Labeling of slices in $h^0$}

Any disc $D^n$ is diffeomorphic to the cone $C(S^{n-1})$ smoothed out over the singular point. Using this fact, it is not hard to see that the cylinder $S^3\times [0,1]$ embeds smoothly into $h^0$ via a map $\psi$ so that $S^3\times \{1\}$ maps to $\partial h^0$ and $S^3\times\{0\}$ maps into the interior of $h^0$. We fix such an embedding $\psi$ once and for all. We label the boundary $\partial h^0$ as $S^3\times \{1\}$ and the image $\psi(S^3\times \{x\})$ in the interior of $h^0$ as $S^3\times\{x\}$ for each $x\in[0,1)$. We recognize that $h^0\setminus \psi(S^3\times(0,1])$ is diffeomorphic to $D^4$. A diagram of this labeling is provided in Figure \ref{figure label slices of h^0}.


\begin{figure}[h]
\includegraphics[width=0.4\textwidth]{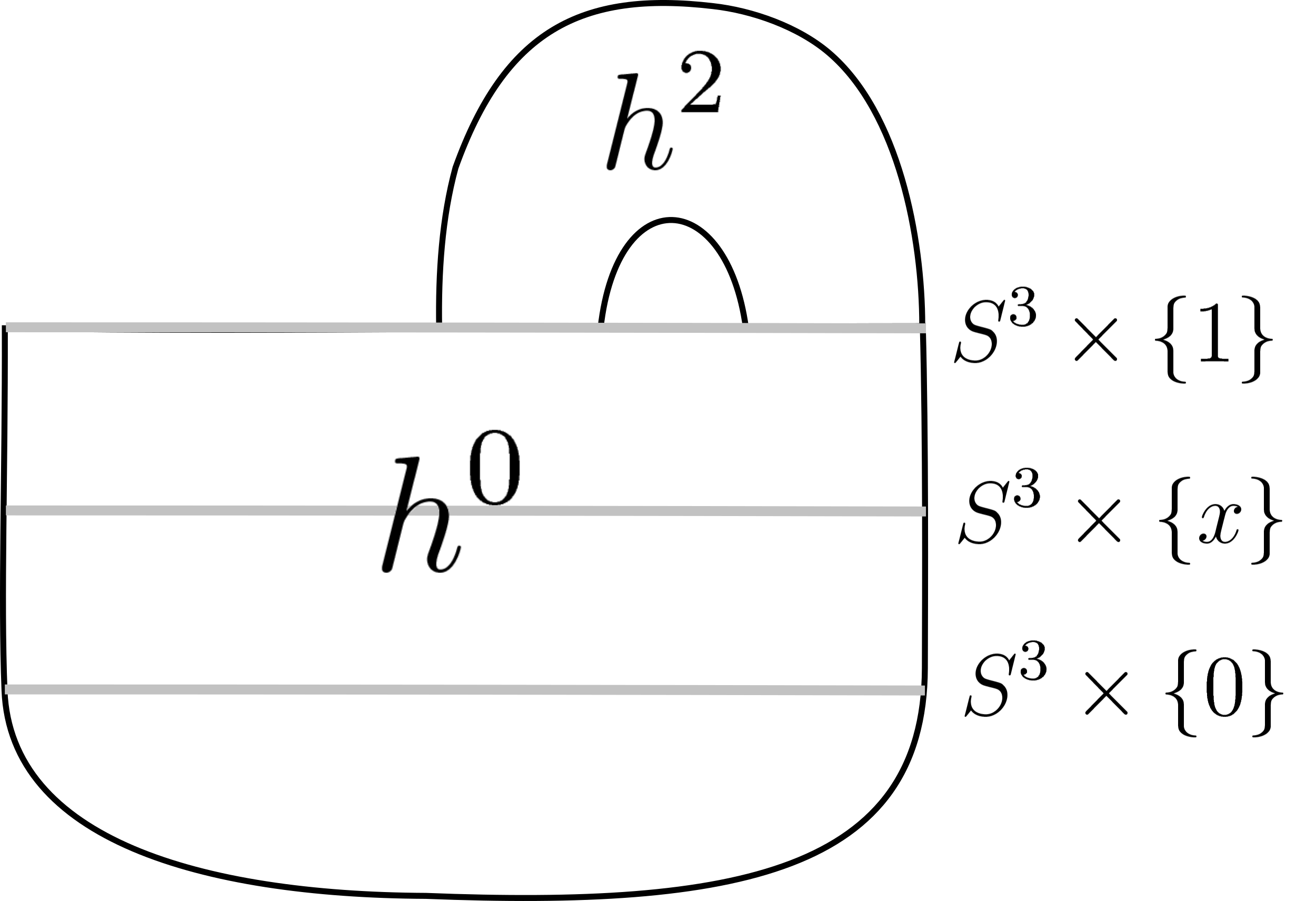} \\
\caption{A flattened (lower-dimensional) handlebody diagram of $\mathbb{C}P^2\setminus B^4$ with labels given by $\psi$.
}\label{figure label slices of h^0}
\end{figure}

\subsubsection{Torus link $T(n,n)$}

The core disc of the 2-handle $h^2$ is $D:=D^2\times 0$. We may push any number of parallel copies $D_1, ..., D_n$ off of $D$ and they will be embedded in $h^2$ without pair-wise or self-intersections. Since $D_1, ..., D_n$ are 2-dimensional discs, their boundaries $\partial D_1, ..., \partial D_n$ are 1-dimensional spheres. Due to the $+1$-twist of the 2-handle attachment, the boundaries $\partial D_1,...., \partial D_n$ form a $T(n,n)$ torus link first in the attaching region $S^1\times D^2$ and then in turn in the boundary $S^3\times\{1\}$ to which $S^1\times D^2$ is attached via $\phi$. A diagram of this is shown in Figure \ref{figure core discs as cups}.


\begin{figure}[h]
\includegraphics[width=0.9\textwidth]{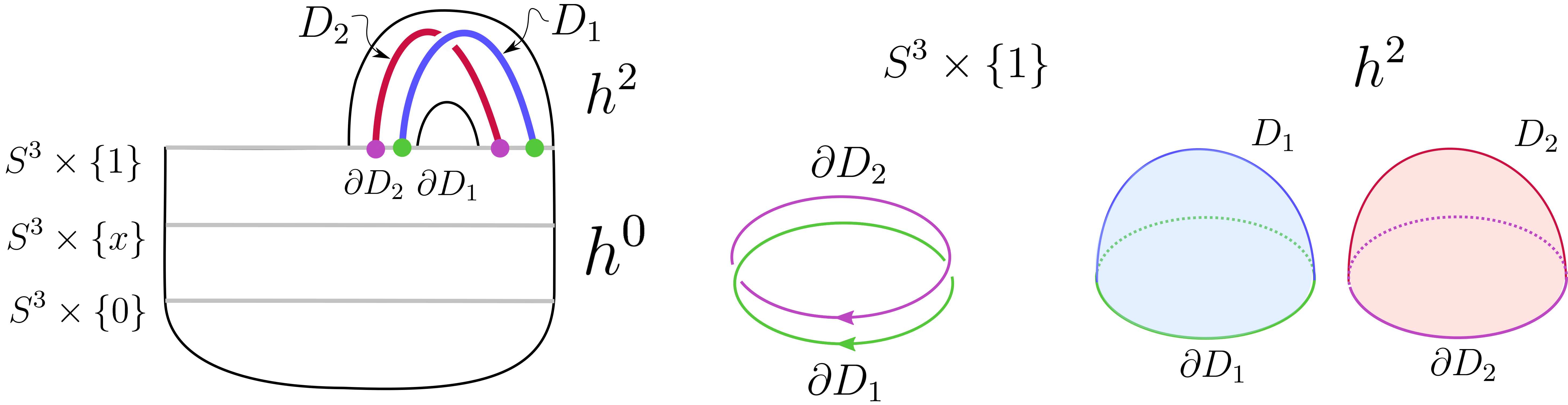}
\caption{(Left) A flattened (lower-dimensional) handlebody diagram of $\mathbb{C}P^2\setminus B^4$. (Center) The link $T(2,2)=\partial D_1\cup \partial D_2$ sitting in the boundary $S^3\times\{1\}$ of $h^0$ with particular orientation. (Right) Copies $D_1, D_2$ of the core disc $D^2\times 0$ of $h^2$ sitting inside $h^2$, pictured in full dimension with unlinked boundaries for the sake of visualization.}\label{figure core discs as cups}
\end{figure}


\begin{figure}[h]
\includegraphics[width=0.32\textwidth]{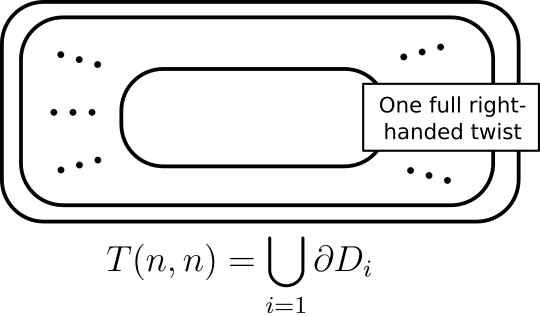}
\caption{The link $T(n,n)$ is unoriented in the diagram. Choice of orientation will determine which coherent band surgeries are possible and the homological degree of a surface with components $D_1,...,D_n$.}\label{linktnn}
\end{figure}

\subsubsection{Homological degree}

When pushing off a parallel copy $D_i$ of the core disc $D$, the orientation may be chosen to be either compatible or or incompatible with the orientation of $D$. The boundary component $\partial  D_i$ of $T(n,n)$ inherits an orientation from $D_i$. Thus, the link $T(n,n)$ has oriented components $\partial D_1, ..., \partial D_n$. All the constructions we perform will involve creating a cobordism $C$ between a knot $K$ and a link $T(n,n)$. Since the link $T(n,n)$ may be capped off with $n$ discs $D_1,...,D_n$, it follows that $S_K=C\cup_{T(n,n)}\bigcup_{i=1}^nD_i$ is a surface in $\mathbb{C}P^2\setminus B^4$ with boundary $K\subset \partial(\mathbb{C}P^2\setminus B^4)$. In particular, $S_K$ represents the class $d\gamma\in H_2(\mathbb{C}P^2\setminus B^4, \partial; \mathbb{Z})$, where $\gamma$ is the generator of $H_2(\mathbb{C}P^2\setminus B^4, \partial; \mathbb{Z})\cong\mathbb{Z}$ represented by $\mathbb{C}P^1$ in $\mathbb{C}P^2\setminus B^4$.

The degree $|[S_K]|=d$ is determined solely by the orientations of the discs $D_1,...,D_n$ and hence by the orientations of the components $\partial D_1,...,\partial D_n$ of $T(n,n)$. The integer $d$ is exactly $d_+-d_-$, where $d_+$ is the number of components of $T(n,n)$ with orientation matching the orientation of $\partial D$ and $d_-$ is the number of components of $T(n,n)$ with orientation opposed to the orientation of $\partial D$.

\subsection{Knots quickly seen to be slice in $\mathbb{C}P^2$}\label{subsection trivial sliceness}

If a knot $K$ satisfies any of the following conditions, then it is slice in $\mathbb{C}P^2$. 
\begin{enumerate}
	\item $g_4(K)=0$.
	\item $u(K)=1$, where $u(K)$ is the unknotting number of $K$.
	\item $K$ is concordant to a knot $J$ such that $g_{\mathbb{C}P^2}(J)=0$.
\end{enumerate}
The author makes no claim of being the first to develop any of these arguments. Exposition is provided only for the sake of readability and reference.

\subsubsection{4-genus as an upper bound}

\begin{lemma}
\label{lemma 4genus upper bound}
Let $K$ be a knot with smooth 4-genus $g_4(K)$. Then
\[g_{\mathbb{C}P^2}(K)\leq g_4(K)\]
\end{lemma}

\begin{proof}
Let $K$ be a knot in $\partial (\mathbb{C}P^2\setminus B^4)$. By ambient isotopy of $S^3$, we can shrink $K$ to be as small as we want until there is a closed $D^4$ neighborhood $N$ around the shrunken $K$. By the definition of the smooth 4-genus, there is a surface $S_K$ of genus $g_4(K)$ in $N$ with $\partial S_K$ being the shrunken $K$. Thus, $K$ bounds a smoothly and properly embedded surface $S_K$ of genus $g_4(K)$ in $\mathbb{C}P^2\setminus B^4$. Since the $\mathbb{C}P^2$-genus of $K$ can only be equal to or smaller than the genus of any surface it bounds, we have the desired inequality.
\end{proof}

\subsubsection{Unknotting number as an upper bound}

\begin{lemma}
\label{lemma unknotting upper bound}
Let $K$ be a knot with unknotting number $u\in \mathbb{Z}_{\geq 0}$. Then
\[g_{\mathbb{C}P^2}(K)\leq \max\{u-1, 0\}\]
\end{lemma}

\begin{proof}
Let $K$ be a knot with unknotting number $u$. If $u=0$, then $K$ is the unknot, which is easily seen to be slice in $\mathbb{C}P^2$. It follows that $g_{\mathbb{C}P^2}(K)\leq \max\{0,u-1\}=0$ since $u-1=-1$. 

We now assume that $u\geq 1$. Fix a sequence of $u$ crossing changes $c_1,...,c_u$ for $K$ that turns $K$ into the unknot. Such a sequence exists by the definition of unknotting number. For $c_1$, we use two parallel copies $D_1$, $D_2$ of the core disc of the 2-handle in $\mathbb{C}P^2$ to realize the crossing change. Specifically, we perform coherent band surgeries between $\partial D_1, \partial D_2$ and $K$ to generate a genus 0 cobordism $C_1$ between $K\subset S^3\times \{1\}$ and $K_1\subset S^3\times\{(u-1)/u\}$, where $K_1$ differs from $K$ only by the crossing change $c_1$. The link  with components $\partial D_1, \partial D_2$ is a $T(2,2)$ link equivalent to either L2a1$\{0\}$ or L2a1$\{1\}$ depending on whether $c_1$ is a change from positive crossing to negative (L2a1$\{0\}$) or negative to positive (L2a1$\{1\}$). The necessary coherent band surgeries are shown in Figure \ref{figure crossing change using handle} and the resulting genus 0 cobordism $C_1$ is shown in Figure \ref{figure crossing change cobordism}. 

There remain $u-1$ necessary crossing changes $c_2, c_3,...,c_u$ to turn $K_1$ into the unknot. Let $K_m$ be the knot obtained from $K_{m-1}$ by performing the crossing change $c_m$. For each remaining $c_i$, where $2\leq i \leq n$, we perform two coherent band surgeries as shown in Figure \ref{figure self surgery}. The first of the surgeries will generate a cobordism $C_{i,a}$ between $K_{i-1}\subset S^3\times\{(u-i+1)/u\}$ and $K_i \# L2a1$ where $L2a1$ is the link $L2a1\{0\}$ or $L2a1\{1\}$ depending on whether $c_i$ is a positive or negative crossing. The second of the surgeries will create a cobordism $C_{i,b}$ between $K_i \# L2a1$ and $K_i\subset S^3\times\{(u-1)/u\}$. The stacking of $C_{i,a}$ and $C_{i,b}$ gives a genus 1 cobordism $C_i$ between $K_{i-1}\subset S^3\times\{(u-i+1)/u\}$ and $K_i\subset S^3\times\{(u-i)/u\}$. The cobordism $C_i$ and the necessary coherent band surgeries needed to create it are shown in Figure \ref{figure self surgery}.

Stacking the cobordisms $C_1, C_2,...C_u$, we obtain a genus $u-1$ cobordism between $K$ and $K_u$. Since $K$ has unknotting number $u$ and we chose a sequence of crossing changes $c_1,...,c_u$ specifically to unknot $K$, it follows that $K_u$ is the unknot. The unknot is easily seen to be capped off with an embedded 2-disc, thus generating a genus $u-1$ properly embedded surface $\bigcup_{i=1}^u C_i$ in $\mathbb{C}P^2\setminus B^4$ with boundary $K$. This shows that $g_{\mathbb{C}P^2}(K)\leq u-1$ and hence $g_{\mathbb{C}P^2}(K)\leq \max\{u-1,0\}=u-1$, as desired.
\end{proof}


\begin{figure}[h]
\includegraphics[width=0.8\textwidth]{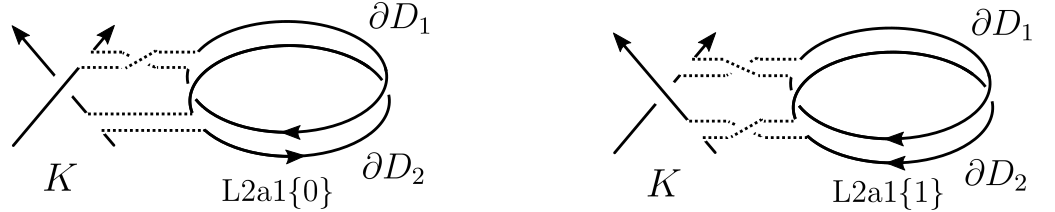}
\caption{A crossing may be changed from positive to negative (left) or from negative to positive (right) using the boundaries $\partial D_1$ and $\partial D_2$. 
}\label{figure crossing change using handle}
\end{figure}


\begin{figure}[h]
\includegraphics[width=0.7\textwidth]{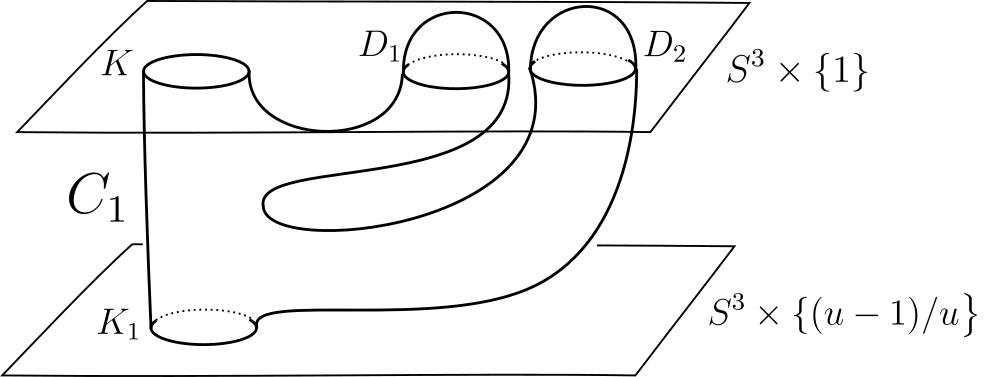}
\caption{The genus 0 cobordism $C_1$ between $K$ and $K_1$. The knot $K_1$ differs from $K$ by a single crossing change.}\label{figure crossing change cobordism}
\end{figure}


\begin{figure}[h]
\includegraphics[width=0.7\textwidth]{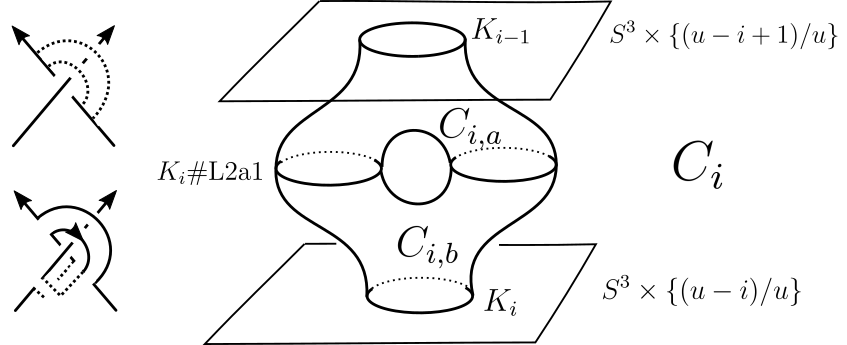}
\caption{Two coherent band surgeries are performed in succession to generate a genus 1 cobordism $C_i$ between $K_{i-1}\subset S^3\times\{(u-i+1)/u\}$ and $K_1\subset S^3\times\{(u-i)/u\}$.}\label{figure self surgery}
\end{figure}

\begin{remark}
Neither Lemma \ref{lemma 4genus upper bound} nor Lemma \ref{lemma unknotting upper bound} are strictly better than the other in terms of finding slice knots in $\mathbb{C}P^2$ quickly. Lemma \ref{lemma unknotting upper bound} is able to detect that $3_1$ is slice in $\mathbb{C}P^2$ but not $8_8$, while Lemma \ref{lemma 4genus upper bound} is able to detect that knot $8_8$ is slice in $\mathbb{C}P^2$ but not $3_1$.
\end{remark}

\subsubsection{Concordance as an upper bound}

\begin{lemma}\label{lemma concordance upper bound}
Let $K$ be a knot in $\partial (\mathbb{C}P^2\setminus B^4)\cong S^3$ and let $K$ be concordant to $J$. Then
\[g_{\mathbb{C}P^2}(K)=g_{\mathbb{C}P^2}(J).\]
\end{lemma}

\begin{proof}
We will show that $g_{\mathbb{C}P^2}(K)\leq g_{\mathbb{C}P^2}(J)$. Equality follows from the symmetry of concordance.

Let $S_J$ be a smoothly and properly embedded surface in $\mathbb{C}P^2\setminus B^4$ with genus $g(S_j)=g_{\mathbb{C}P^2}(J)$ and $\partial S_J=J\subset S^3\times\{0\}$. By definition of concordance \cite{livnai}, there is a genus 0 cobordism $C$ between $J$ in $S^3\times \{0\}$ and $K$ in $S^3\times \{1\}$. The surface $S_K=C\cup_J S_J$ is bounded by $K$ and has genus $g_{\mathbb{C}P^2}(J)$. The desired inequality follows.
\end{proof}

\subsection{Proof of Theorem \ref{mainresult1}}\label{subsection proof of mainresult1}

\begin{proof}[Proof of Theorem \ref{mainresult1}]
Let $L$ be a link that is listed in the statement of the theorem. Let $K$ be a knot such that $mK$, and hence $rmK$, is obtained from $L$ via a coherent band surgery.

Every knot is concordant to itself. By Theorem 3.3.2 \cite{livnai}, since $K$ is concordant to itself, there is a genus 0 cobordism $C_K$ between $K\subset S^3\times \{1\}$ and $K\subset S^3\times \{0\}$. The link $L\in S^3\times \{1/2\}$ is obtained from some link $T(n,n)=\cup_{i=1}^n D_i\subset S^3\times \{1\}$ via a series of coherent band surgeries such that the resulting cobordism $C_{L,a}$ between $T(n,n)$ and $L$ has only Morse critical points of the form $-x_1^2+x_2^2$. A proof of this for each specific link in the statement of the theorem is given diagrammatically in Appendix \ref{appendix surgeries for MR1}. By hypothesis, there is a genus 0 cobordism $C_{L,b}$ between $L\subset S^3\times \{1/2\}$ and $rmK\subset S^3\times\{0\}$. Since the cobordism $C_{L,b}$ only has a single Morse critical point of the form $-x_1^2+x_2^2$, stacking $C_{L,a}$ and $C_{L,b}$ to obtain $C_L=C_{L,a}\cup_L C_{L,b}$ will not introduce any genus. We cap off $C_L$ with the core discs $D_1,...,D_n$ in the 2-handle $h^2$ to obtain a disc $C_{rmK}$ bounded by $rmK\subset S^3\times\{0\}$. We take the boundary connect sum $C_K\natural C_{rmK}$ over the $K$ and $rmK$ boundary components in $S^3\times \{0\}$. We now have a surface $C_K\natural C_{rmK}$ with boundary components $K\subset S^3\times\{1\}$ and $K\# rmK\subset S^3\times \{0\}$ as shown in Figure \ref{figure boundary connect sum surface}. By Theorem 3.1.1 of \cite{livnai}, $K\#rmK$ is a slice knot. Thus, we may cap off the $K\# rmK$ boundary component of $C_K\natural C_{rmK}$ with a disc $D'$ to get $S=C_K\natural C_{rmK} \cup_{K\# rmK} D'$. The surface $S$ is orientable, is smoothly and properly embedded in $\mathbb{C}P^2\setminus B^4$, has no genus, and has boundary $K$. Thus, $K$ is slice in $\mathbb{C}P^2$.
\end{proof}


\begin{figure}[h]
\includegraphics[width=0.9\textwidth]{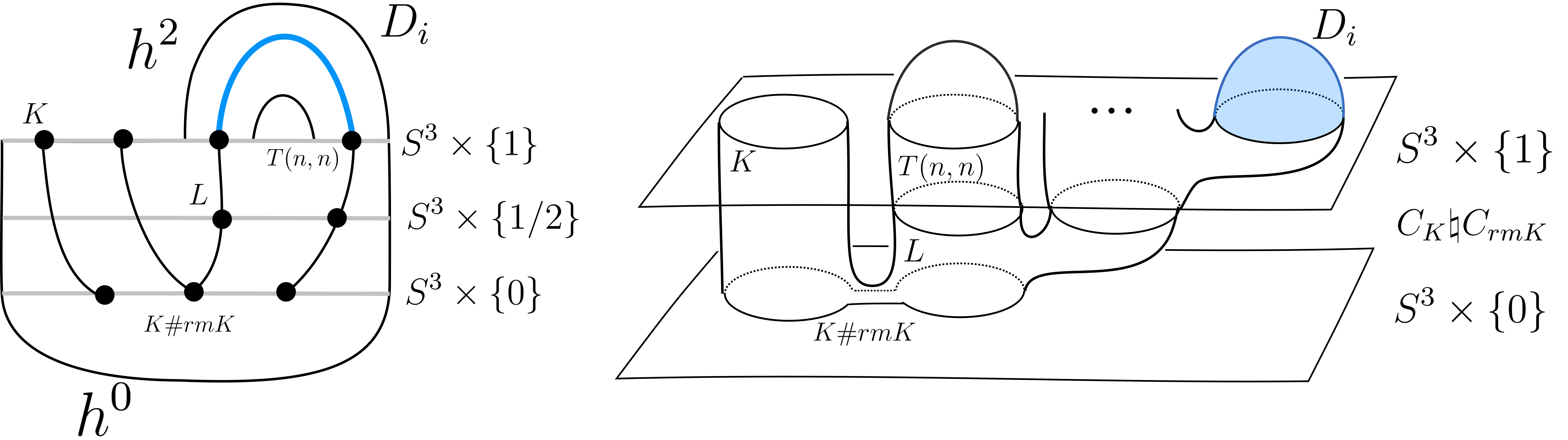}
\caption{$C_K\natural C_{rmK}$ is the cobordism between $K\subset S^3\times\{1\}$ and $K\#rmK\subset S^3\times \{0\}$.}\label{figure boundary connect sum surface}
\end{figure}

\section{Lower bounds}\label{section lower bounds}

\subsection{Obstructing homological degrees}

We adapt several results from low-dimensional topology to the world of knots and surfaces in $\mathbb{C}P^2\setminus B^4$. In particular, we rely on Corollary \ref{corollary ozvath szabo}, Corollary \ref{corollary kronheimer mrowka}, Corollary \ref{corollary gilmer viro}, Corollary \ref{corollary lawson}, and Corollary \ref{corollary robertello} to obstruct all possible slice degrees  for knots with certain characteristics. We list the utility of each Corollary and remind the reader that we are considering slice degrees in absolute value. For example, if we say that degree $5$ is obstructed, we really mean that both degrees $5$ and $-5$ are obstructed.

\begin{itemize}
	\item Corollary \ref{corollary ozvath szabo}. Positive signatures of alternating knots obstruct small slice  degrees.
	\item Corollary \ref{corollary kronheimer mrowka}. The smooth 4-genus obstructs large slice degrees.
	\item Corollary \ref{corollary gilmer viro}. The signature obstructs almost all even slice degrees.
	\item Corollary \ref{corollary lawson}. Knots with odd slice degree obstruct their mirrors from having odd slice degrees.
	\item Corollary \ref{corollary robertello}. The Arf invariant obstructs half the remaining odd slice degrees.
\end{itemize}

\subsubsection{Positive signatures of alternating knots obstruct small degrees}

%
%

Theorem \ref{theorem ozvath szabo} makes use of Ozvath and Szabo's Tau invariant, which is derived from Knot Floer homology. For our calculations it is only important to note that when $K$ is an alternating knot, the Tau invariant $\tau(K)$ can be  expressed in terms of it's signature $\sigma(K)$. More on the Tau invariant can be found in Ozvath and Szabo's paper \cite{Ozvath2003}.


\begin{theorem}[Ozvath, Szabo \cite{Ozvath2003}]
\label{theorem ozvath szabo}
Let $W$ be a smooth, oriented four-manifold with $b_2^+(W)=b_1(W)=0$ and $\partial W=S^3$. If $S_K$ is any properly embedded surface in $W$ such that $\partial S_K = K$ for some knot $K$, then
\begin{align*}
2\tau (K)+\bigl| [S_K] \bigr| + [S_K]\cdot[S_K]\leq 2g(S_k)
\end{align*}
\end{theorem}


\begin{lemma}\label{lemma ozvath szabo}
Let $K$ be an alternating knot that bounds a properly embedded surface $S_K$ in $\mathbb{C}P^2\setminus B^4$ with $[S_K]=d\gamma\in H_2(\mathbb{C}P^2\setminus B^4,\partial;\mathbb{Z})$. Then
\[g(S_K)\geq \frac{\sigma(K)}{2}+\frac{|d|(1-|d|)}{2}\].
\end{lemma}

\begin{proof}
Let $S_K$ be a properly embedded surface in $\mathbb{C}P^2\setminus B^4$ bounded by $K$ with $[S_K]=d\gamma\in H_2(\mathbb{C}P^2\setminus B^4,\partial;\mathbb{Z})$. Then $rmK$ bounds a surface $\overline{S_K}$ in $\overline{\mathbb{C}P^2}\setminus B^4$ with $[\overline{S_K}]= d\overline{\gamma} \in H_2(\overline{\mathbb{C}P^2}\setminus B^4,\partial;\mathbb{Z})$. By Theorem \ref{theorem ozvath szabo},
\begin{align}
g(\overline{S_K}) &\geq \tau(rmK)+\frac{|[\overline{S_K}]|+[\overline{S_K}]\cdot[\overline{S_K}]}{2} \nonumber \\
g(\overline{S_K}) &\geq \tau(rmK)+\frac{|d|-|d|^2}{2} \nonumber \\
g(\overline{S_K}) &\geq \tau(rmK)+\frac{|d|(1-|d|)}{2} \label{eq1}
\end{align}
Ozvath, Szabo proved that $\tau(K)=\frac{-\sigma(K)}{2}$ for alternating knots \cite{Ozvath2003}. Further, they showed that $-\tau(K)=\tau(rmK)$. Clearly, $g(S_K)=g(\overline{S_K})$. Making these substitutions into (\ref{eq1}) gives the desired inequality. \\
\end{proof}


\begin{corollary}
\label{corollary ozvath szabo}
Let $K$ be an alternating knot with signature $\sigma(K)$ that bounds a properly embedded 2-disc $D_K$ in $\mathbb{C}P^2\setminus B^4$ with $[D_K]=d\gamma\in H_2(\mathbb{C}P^2\setminus B^4,\partial;\mathbb{Z})$.
\begin{enumerate}
	\item If $\sigma(K)\geq 2$, then $|d|\notin\{0,1\}$,
	\item If $\sigma(K)\geq 4$, then $|d|\notin \{0,1,2\}$
\end{enumerate}
\end{corollary}

\begin{proof}
We prove (2). Part (1) is analogous. Suppose $K$ is an alternating knot with $\sigma(K)\geq 4$ that bounds a properly embedded 2-disc $D_K$ in $\mathbb{C}P^2\setminus B^4$ such that $[D_K]=d\gamma\in H_2(\mathbb{C}P^2\setminus B^4)$. By Lemma \ref{lemma ozvath szabo},
\begin{align*}
0		   	&\geq \sigma(K)+|d|(1-|d|) \\
|d|(|d|-1) 	&\geq \sigma(K) \\
|d|(|d|-1) 	&\geq 4 \\
|d|^2-|d|-4 &\geq 0.
\end{align*}
By simple algebra, $|d|\geq \frac{1}{2}+\frac{1}{2}\sqrt{1+4\cdot 4}>\frac{1}{2}+2=\frac{5}{2}$.
\end{proof}

\subsubsection{The smooth 4-genus obstructs large degrees.}

%
%


\begin{theorem}[Kronheimer, Mrowka \cite{Kronheimer1994}]
\label{theorem kronheimer mrowka}
Let $S$ be an oriented 2-manifold smoothly embedded in $\mathbb{C}P^2$ such that $[S]=d\gamma\in H_2(\mathbb{C}P^2;\mathbb{Z})$ with $d\geq 0$ and $g$ is the genus of $S$. Then $2g\geq (d-1)(d-2)$.
\end{theorem}

The reason for the $d\geq 0$ condition is because the Thom Conjecture associates the embedded surface $S$ with an algebraic curve. The degree of an algebraic curve may not be negative.


\begin{lemma}
\label{lemma kronheimer mrowka}
Let $K$ be a knot that bounds a properly embedded 2-disc $D_K$ in $\mathbb{C}P^2\setminus B^4$ with $[D_K]=d\gamma\in H_2(\mathbb{C}P^2\setminus B^4,\partial;\mathbb{Z})$ and $d\geq 0$. Then
\[2g_4(K)\geq (d-1)(d-2)\]
where $g_4(K)$ is the slice genus of $K$.
\end{lemma}

\begin{proof}
Suppose that $K$ bounds a properly embedded 2-disc $D_K$ in $\mathbb{C}P^2\setminus B^4$ with $[D_K]=d\gamma\in H_2(\mathbb{C}P^2\setminus B^4,\partial;\mathbb{Z})$, where $d\geq 0$. By definition, $mK$ bounds an orientable surface $S_{mK}$ of genus $g_4(K)$ in $D^4$. We glue $D^4$ to $\mathbb{C}P^2\setminus B^4$ via an orientation-reversing diffeomorphism $\phi:\partial D^4 \hookrightarrow \partial (\mathbb{C}P^2\setminus B^4)$ that identifies $K$ to the image $\phi(mK)=K$. In doing so, we obtain an embedded closed surface $S=D_k\cup_K S_{mK}$ in $\mathbb{C}P^2$ with $g(S)=g(D_K)+g(S_{mK})=g(S_{mk})=g_4(K)$. By Theorem \ref{theorem kronheimer mrowka}, we have the desired inequality for $d\geq 0$.
\end{proof}

If we are ever able to obstruct slice degree $d$ for a knot, then we also have an obstruction to the slice degree $-d$. The argument for this follows:

Let $K$ be a knot such that it cannot bound a properly embedded disc in $\mathbb{C}P^2\setminus B^4$ with degree $d\in\mathbb{Z}$. Now suppose for contradiction that $K$ bounds a properly embedded disc $D_K$ in $\mathbb{C}P^2\setminus B^4$  with $[D_K]=-d\gamma\in H_2(\mathbb{C}P^2\setminus B^4, \partial; \mathbb{Z})$. By changing the string-orientation of $K$, we get $rK$. This changes the orientation of $D_K$ and thereby the sign of $[D_K]$. We now have that $rK$ bounds a properly embedded disc $\overline{D_K}$ with $[\overline{D_K}]=-[D_K]=-(-d)=d\gamma\in H_2(\mathbb{C}P^2\setminus B^4, \partial; \mathbb{Z})$. Forgetting about the orientation of $rK$, we just have an unoriented $K$. Now we have that $K$ bounds a properly embedded disc $\overline{D_K}$ with degree $d$, a contradiction.


\begin{corollary}
\label{corollary kronheimer mrowka}
Let $K$ be a knot with smooth four genus $g_4(K)$ that bounds a properly embedded disc $D_K$ in $\mathbb{C}P^2\setminus B^4$ with $[D_K]=d\gamma\in H_2(\mathbb{C}P^2\setminus B^4,\partial;\mathbb{Z})$. 
\begin{enumerate}
	\item If $g_4(K)\leq 2$, then $d\in\{0,1,2,3\}$,
	\item If $g_4(K) \leq 5$, then $d\in\{0,1,2,3,4\}$.
\end{enumerate}
\end{corollary}

\subsubsection{The signature obstructs almost all even degrees.}

%
%


\begin{theorem}[Gilmer \cite{Gilmer1981}, Viro \cite{Viro1970}, Yasuhara \cite{Yasuhara1996}]\label{theorem gilmer viro}
Let $K$ be a knot in $\partial (\mathbb{C}P^2\setminus B^4)\cong S^3$. Suppose that $K$ bounds a properly embedded surface $S_K$ in $\mathbb{C}P^2\setminus B^4$ and $[S_K]=d\gamma\in H_2(\mathbb{C}P^2\setminus B^4,\partial;\mathbb{Z})$ is divisible by 2. Then
\begin{align*}
4g(S_K)+2			&\geq \Bigl| d^2 - 2 - 2\sigma(K) \Bigr| \\
\end{align*}
\end{theorem}


\begin{corollary}
\label{corollary gilmer viro}
Let $K$ be a knot with signature $\sigma(K)$.
\begin{enumerate}
	\item If $\sigma(K)\leq -4$ or $\sigma(K)=4$ then $K$ does not bound a properly embedded 2-disc in $\mathbb{C}P^2\setminus B^4$ with even degree,
	\item If $\sigma(K)=-2$ and $K$ bounds a properly embedded 2-disc in $\mathbb{C}P^2\setminus B^4$ of even degree $d$, then $d=0$,
	\item If $\sigma(K)=0$ and $K$ bounds a properly embedded 2-disc in $\mathbb{C}P^2\setminus B^4$ of even degree $d$, then $|d|\in \{0,1\}$,
	\item If $\sigma(K)=2$ and $K$ bounds a properly embedded 2-disc in $\mathbb{C}P^2\setminus B^4$ of even degree $d$, then $|d|=2$,
	\item If $\sigma(K)=6$ or $\sigma(K)=8$ and  $K$ bounds a properly embedded 2-disc in $\mathbb{C}P^2\setminus B^4$ of even degree $d$, then $|d|=4$
\end{enumerate}
\end{corollary}

\begin{proof}
Substitute $g(S_K)=0$ into Theorem \ref{theorem gilmer viro} and consider the case for each signature.
\end{proof}

\subsubsection{Knots with odd slice degree obstruct their mirrors from having odd slice degrees.}

%
%


\begin{theorem}[Lawson \cite{Lawson1992}]
\label{theorem lawson}
Let $S$ be a characteristic embedded 2-sphere in $2\mathbb{C}P^2\# \overline{\mathbb{C}P^2}$ (respectively $\mathbb{C}P^2\# 2\overline{\mathbb{C}P^2}$). Then $[S]\cdot[S]=1$ (respectively $[S]\cdot [S]=-1$).
\end{theorem}

It is worth noting that an embedded surface $S$ in $m\mathbb{C}P^2\# n\overline{\mathbb{C}P^2}$ represented by class $[S] = \Sigma_{i=1}^m d_i\gamma_i +\Sigma_{j=1}^n d_j'\overline{\gamma}_j\in H_2(m\mathbb{C}P^2\#n\overline{C}P^2;\mathbb{Z})$ where $\gamma_i\cdot\gamma_i=1$ for all $1\leq i\leq n$ and $\overline{\gamma}_j\cdot\overline{\gamma}_j=-1$ for all $1\leq j\leq n$ is characteristic if and only if $d_i$ and $d_j'$ are odd for all $1\leq i\leq m$ and $1\leq j\leq n$.


\begin{corollary}\label{corollary lawson}
Let $K$ be a knot that bounds a properly embedded disc $D_K$ in $\mathbb{C}P^2\setminus B^4$ with $[D_K]=d\gamma\in H_2(\mathbb{C}P^2\setminus B^4,\partial;\mathbb{Z})$.
\begin{enumerate}
	\item Let $d$ be of odd degree and $|d|\geq 3$. Then $mK$ does not bound a properly embedded 2-disc in $\mathbb{C}P^2\setminus B^4$ of odd degree.
	\item Let $|d|=1$. If $mK$ bounds a properly embedded disc $D_K'$ in $\mathbb{C}P^2\setminus B^4$ with $[D_K']=d'\gamma\in H_2(\mathbb{C}P^2\setminus B^4,\partial;\mathbb{Z})$ of odd degree $d'$, then $|d'|=1$. 
\end{enumerate}
\end{corollary}

\begin{proof}
(1) Suppose for contradiction that $mK$ bounds a properly embedded 2-disc $D_K'$ in $\mathbb{C}P^2\setminus B^4$ with $[D_K']=d'\gamma_2\in H_2(\mathbb{C}P^2\setminus B^4,\partial;\mathbb{Z})$ and $d'$ odd. By taking a single parallel copy $\widetilde{D_K}$ of the core disc in the $(-1)$-twisted 2-handle of $\mathbb{C}P^2\#\overline{\mathbb{C}P^2}\setminus B^4$ we may perform surgery to absorb $\widetilde{D_K}$ into $D_K'$. This gives us that $mK$ bounds a properly embedded 2-disc $D_K''$ in $\mathbb{C}P^2\#\overline{\mathbb{C}P^2}\setminus B^4$ with $[D_K'']=d'\gamma_2+\overline{\gamma}_3\in H_2(\mathbb{C}P^2\#\overline{\mathbb{C}P^2}\setminus B^4)$. Gluing $\mathbb{C}P^2\setminus B^4$ to $\mathbb{C}P^2\#\overline{\mathbb{C}P^2}\setminus B^4$ via an orientation reversing diffeomorphism,  we have an embedded characteristic 2- sphere $S=D_K\cup_K D_K''$ in $2\mathbb{C}P^2\#\overline{\mathbb{C}P^2}$ with $[S]=d\gamma_1+d'\gamma_2+\overline{\gamma}_3\in H_2(2\mathbb{C}P^2\#\overline{\mathbb{C}P^2};\mathbb{Z})$. By Theorem \ref{theorem lawson},
\begin{align*}
1			&= [D_K]\cdot[D_K] \\
			&= d^2+(d')^2-1 \\
			&\geq 9+(d')^2-1 \\
-7 			&\geq (d')^2
\end{align*}
a contradiction.

(2) Following the proof of (1) above, we have
\begin{align*}
1			&= [D_K]\cdot[D_K] \\
			&= 1+d^2-1 \\
1 			&= d^2
\end{align*}
as desired.
\end{proof}

\subsubsection{The Arf invariant obstructs half the remaining odd slice degrees.}

%
%

\begin{definition}[Robertello \cite{Robertello1965}]\label{definition robertello arf}
Let $f:S^2\rightarrow M^4$ be a combinatorial embedding of the 2-sphere $S^2$ into a closed, oriented, simply connected, differentiable 4-manifold $M^4$. Let $f$ be differentiable and regular except at one point $x_0\in S^2$, and suppose there exists a differentiably embedded 4-disk $D^4$ such that $D^4\subset M^4$, $f(x_0)$ (the singularity of $f$) is at the center of $D^4$, and $f(S^2)\cap D^4$ is a knot in $S^3=\partial D^4$. Let $\xi = [f(S^2)]\in H_2(M^4,\mathbb{Z})$ be characteristic. Then
\[Arf(K)\equiv \frac{\xi\cdot \xi - \sigma(M^4)}{8}\mod 2.\]
\end{definition}

\begin{lemma}
Let $K$ be a knot with Arf invariant Arf$(K)$ and let $D_K$ be a properly embedded disc in $\mathbb{C}P^2\setminus B^4$ with $\partial D_K = K$ and odd homological degree $d$. Then
\[\text{Arf}(K)\equiv\frac{d^2-1}{8}\mod 2.\]
\end{lemma}

\begin{proof}
Suppose that $K$ is a knot  in  $\partial (\mathbb{C}P^2\setminus B^4)\cong S^3$ and $K$ bounds a properly embedded disc $D_K$ with degree $d$. We may cap off $\mathbb{C}P^2\setminus B^4$ with a $D^4$ and $D_K$ with the cone $C(K)$ of $K$, as shown in Figure \ref{figure robertello cone}. This gives us an embedding of $S^2$ into $\mathbb{C}P^2$ with only one singularity point, the cone point of $C(X)$. Call the image of this embedding $S$, which is the union over $K$ of $D_K$ and $C(K)$. Since the degree $d$ of $[S]\in H_2(\mathbb{C}P^2;\mathbb{Z})$ is determined completely by the number of times it traverses the 2-handle $h^2$ of $\mathbb{C}P^2$, the degree of $[S]$ is equal to the degree $d$ of $[D_K]\in H_2(\mathbb{C}P^2\setminus B^4)$. Since $d$ is odd, it means that $[D_K]$ is characteristic. By substitution into definition \ref{definition robertello arf}, we have the desired result.
\end{proof}


\begin{figure}[h]
\includegraphics[width=0.55\textwidth]{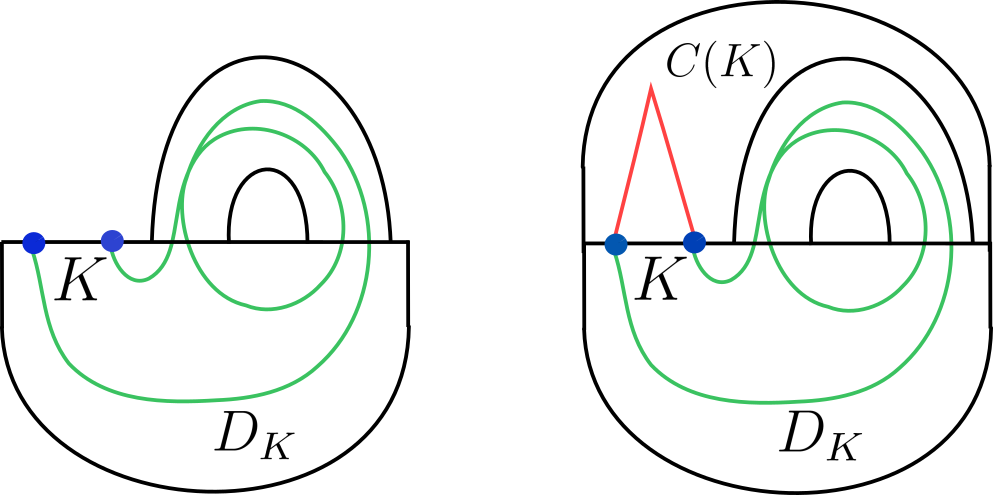} \\
\caption{(Left) A knot $K$ in $S^3=\partial \mathbb{C}P^2\setminus B^4$ bounding a properly embedded disc $D_K$. (Right) $\mathbb{C}P^2\setminus B^4$ capped off with a $D^4$ and $D_K$ capped off with the cone $C(K)$.
}\label{figure robertello cone}
\end{figure}

\begin{corollary}\label{corollary robertello}
Let $K$ be a knot and $D_K$ a properly embedded disc in $\mathbb{C}P^2\setminus B^4$ with characteristic (odd) degree $d$.
\begin{enumerate}
	\item If Arf$(K)=0$, then $|d|$ is not 3.
	\item If Arf$(K)=1$, then $|d|$ is not 1.
\end{enumerate}
\end{corollary}

\subsection{Proofs of Theorems \ref{mainresult2} and \ref{mainresult3}}

%
%

With Corollaries \ref{corollary ozvath szabo}, \ref{corollary kronheimer mrowka}, \ref{corollary gilmer viro}, \ref{corollary lawson}, and \ref{corollary robertello}, we may now easily prove Theorems \ref{mainresult2} and \ref{mainresult3}.

\begin{proof}[Proof of Theorem \ref{mainresult2}]
Let $K$ be an alternating knot with $\sigma(K)=4$, $g_4(K)\leq 2$, and Arf$(K)=0$. Suppose for contradiction that $K$ bounds a properly embedded disc $D_K$ in $\mathbb{C}P^2\setminus B^4$ with slice degree $d$. By Corollary \ref{corollary kronheimer mrowka},  $|d|\in\{0,1,2,3\}$. By Corollary \ref{corollary ozvath szabo}, $|d|$ is not 0, 1, or 2. By Corollary \ref{corollary robertello}, $|d|$ is not 3. Since each properly embedded surface in $\mathbb{C}P^2\setminus B^4$ has some homological degree, we have reached a contradiction. It follows that such a disc $D_K$ cannot exist.
\end{proof}

\begin{proof}[Proof of Theorem \ref{mainresult3}]
Let $K$ be an alternating knot with $|\sigma(K)|=4$. Suppose for contradiction that both $K$ and $mK$ bound properly embedded discs $D_K$ and $D_K'$ in $\mathbb{C}P^2\setminus B^4$ respectively with slice degrees $d$ and $d'$. Without loss of generality, suppose that $\sigma(K)=4$ and $\sigma(-K)=-4$. By Corollary \ref{corollary gilmer viro} we know that $d$ and $d'$ are not even. Hence both $d$ and $d'$ are odd. By Corollary \ref{corollary ozvath szabo}, we know that $d$ cannot be 1. Thus, $d$ is an odd number greater than or equal to 3. It follows by Corollary \ref{corollary lawson} that $d'$ is not odd, a contradiction.
\end{proof}

\section{Additional computations}\label{section computation}

In this section we provide computations of the $\mathbb{C}P^2$-genus for knots that were beyond the intended scope of this work. Namely, we compute the $\mathbb{C}P^2$-genus for a finite set of prime knots of 9- and 10-crossings and show an infinite family $\{K_n\}$ of knots such that $K_n$ and $mK_n$ have differing $\mathbb{C}P^2$-genus for each $n\in\mathbb{N}$.

\subsection{Genera computed using Theorem \ref{mainresult1}}

One can show that a knot $K$ is slice in $\mathbb{C}P^2$ by showing a coherent band surgery taking $mK$ to one of the links listed in Theorem \ref{mainresult1} or vice versa (taking one of the links to $mK$).. We have done this for the following prime knots of $9-$ and $10-$crossings
\[m9_4,\hspace{.2cm} 9_5,\hspace{.2cm} m9_{13},\hspace{.2cm} m9_{15},\hspace{.2cm} 9_{29},\hspace{.2cm} m9_{35},\hspace{.2cm} m10_{11},\hspace{.2cm} m10_{12},\hspace{.2cm} 10_{37}.\]

The required coherent band surgeries are provided in Appendix \ref{appendix surgeries knots to links}. In addition to the finite list of knots above, we have computed the $\mathbb{C}P^2$-genus of an infinite family $\{K_n\}$ of knots along with their mirrors. Consider the family $\{K_n\}_{n\geq 1}$ shown in Figure \ref{figure K_n and band move}. In the same figure, we see the coherent band surgery required to take each $K_n$ to the link L4a1$\{1\}$. Thus, by Theorem \ref{mainresult1}, $mK_n$ is slice in $\mathbb{C}P^2$ for each $n\in\mathbb{N}$.


\begin{figure}[h]
\includegraphics[width=0.25\textwidth]{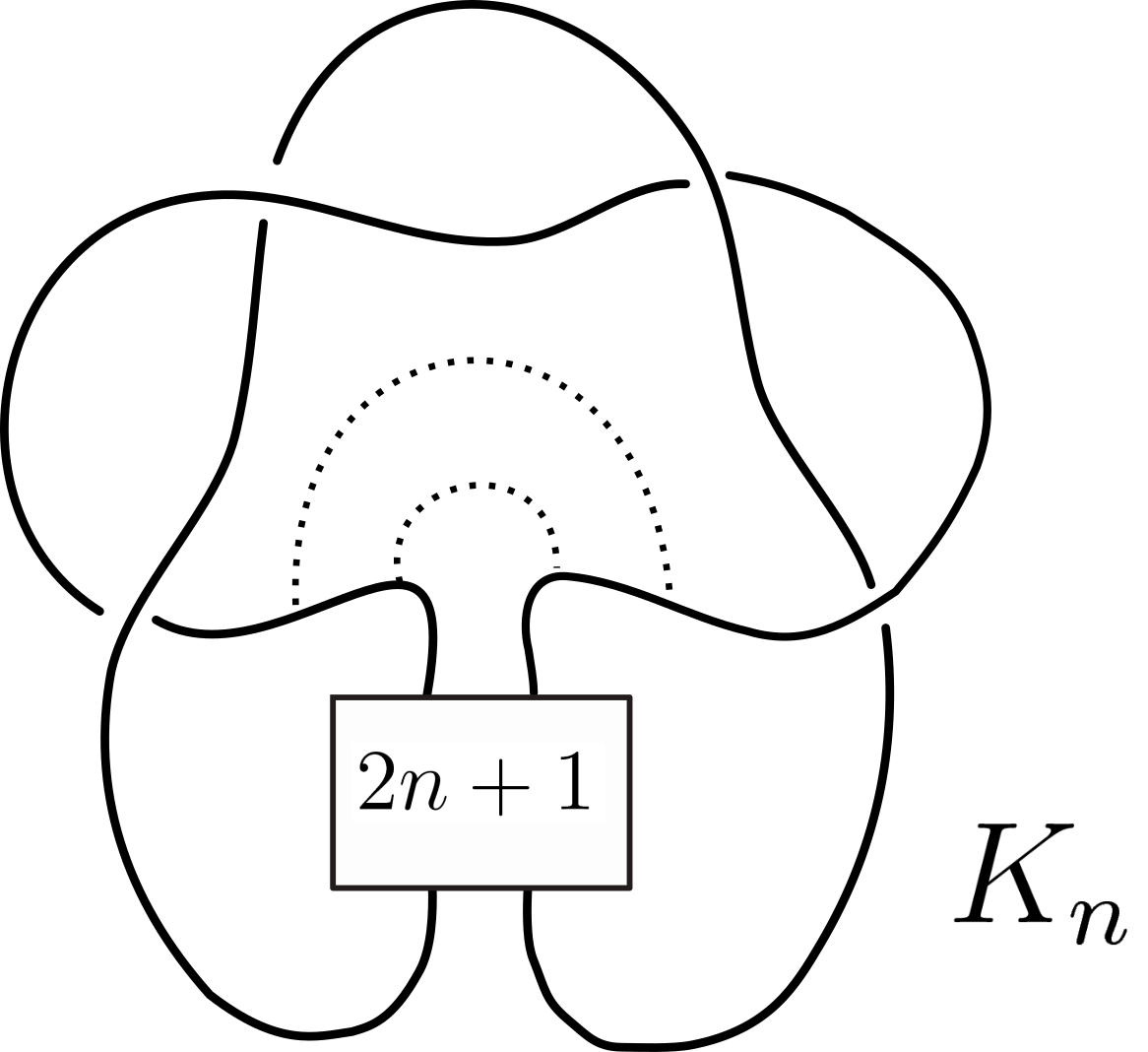}
\caption{Coherent band surgery between $K_n$ and L4a1$\{1\}$.}\label{figure K_n and band move}
\end{figure}

\subsection{Genera computed using Theorem \ref{mainresult2}}

From the preceeding subsection we found the knots $m9_4$ and $m9_{13}$ to be slice in $\mathbb{C}P^2$. It follows from Theorem \ref{mainresult2} that their mirrors
\[9_4,\hspace{.2cm}, 9_{13},\hspace{.2cm}\]
are not slice in $\mathbb{C}P^2$. The author has constructed an explicit genus 1 surface in $\mathbb{C}P^2\setminus B^4$ bounding each knot. It follows that the $\mathbb{C}P^2$-genus of both knots is 1.

We may also apply Theorem \ref{mainresult2} to the infinite family $\{K_n\}$. By inspection, one may check that for any $n\geq 1$, the unknotting number of $K_n$ is 2. That is, no single crossing change will turn $K_n$ into the unknot, but one can always find two that will. Thus $u(K_n)=2$ and by Lemma \ref{lemma unknotting upper bound}, we have that $g_{\mathbb{C}P^2}(K_n)\leq 1$. To determine the explicit $\mathbb{C}P^2$-genus of $K_n$ it is left to show that $g_{\mathbb{C}P^2}(K_n)\geq 1$. To use Theorem \ref{mainresult2}, we need to show that each $K_n$ is alternating and $|\sigma(K_n)|=4$. 

That $K_n$ is alternating may be confirmed by an informal combination of inspection and induction. Starting with $n=1$, we see that $K_1$ (which is the knot $7_3$) is alternating. Moving from $K_n$ to $K_{n+1}$ we see that two crossings are added in such a way that it maintains the alternating nature. 

In order to compute the signature of $K_n$, we use the method of \cite{livnai}. That is, we compute a Seifert matrix $A$ for $K_n$ and determine the signature of $K$ to be the the signature of the symmetric matrix $A+A^T$.


\begin{figure}[h]
\includegraphics[width=0.6\textwidth]{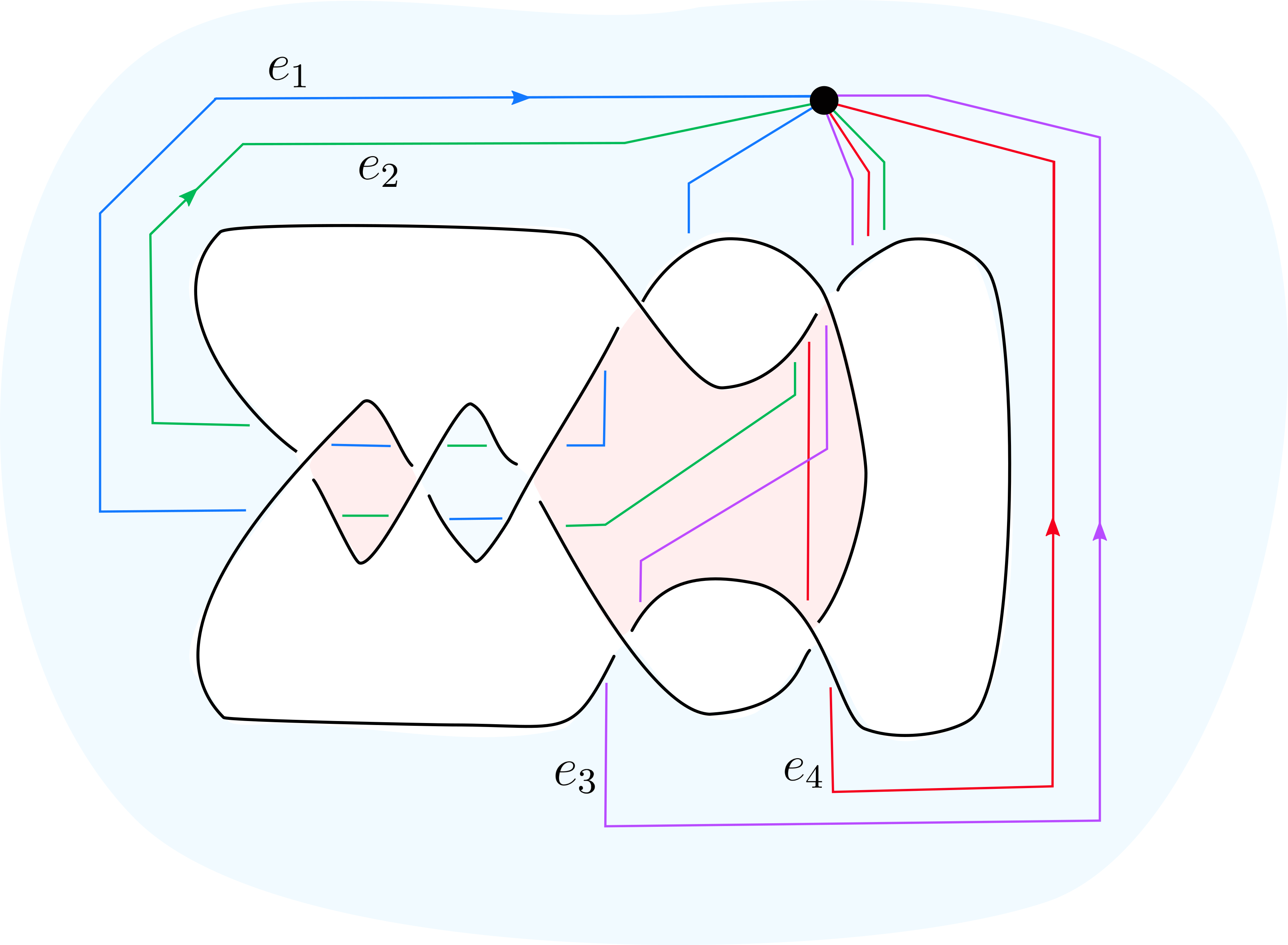}
\caption{The Seifert surface $F_1$ and a choice of generators $e_1,...,e_4$ of $H_1(F_1,\partial F_1;\mathbb{Z})$.}\label{seifert7_3}
\end{figure}

For simplicity we start with $n=1$. For reference $K_1$ is the positive knot $7_3$. We choose the Seifert surface $F_1$ and generators $e_1, e_2, e_3, e_4$ for $H_1(F_1,\partial F_1;\mathbb{Z})$ as shown in Figure \ref{seifert7_3}. The Seifert matrix associated to $F_1$ is

\[
A_1=
  \begin{bmatrix}
    -2 & -2 & 0 & 0 \\
    -1 & -2 & 0 & 0 \\
    0 & 0 & -1 & 0 \\
    0 & 0 & -1 & -1 
  \end{bmatrix}
\]
\\

The matrix $A_1+A_1^T$ is Hermitian, so one can easily check by Sylvester's criteria that $A_1+A_1^T$ is negative-definite and therefore has signature $-4$. It follows from \cite{livnai} that $\sigma(K_1)=-4$. Now let $n$ be arbitrary. Using the same style of Seifert surface, calling it $F_n$, as in Figure \ref{seifert7_3}, and same style of generators $e_1,...,e_4$, then the corresponding Seifert matrix to $F_n$ will be

\[
A_n=
  \begin{bmatrix}
    -1-n & -1-n & 0 & 0 \\
    -n & -1-n & 0 & 0 \\
    0 & 0 & -1 & 0 \\
    0 & 0 & -1 & -1 
  \end{bmatrix}
\]
\\

Again, by checking Sylvester's criteria, we find that $A_n+A_n^T$ is negative definite and hence $\sigma (K_n)=-4$.

\subsection{Genera computed using Theorem \ref{mainresult3}}

A search of Knotinfo \cite{knotinfo} shows that there are 126 alternating prime knots up to 12-crossings with signature 4, smooth four genus less than 3, Arf invariant 0, and unknotting number 2. By Theorem \ref{mainresult3} and Lemma \ref{lemma unknotting upper bound}, these knots all have $\mathbb{C}P^2$-genus equal to 1. Beyond the prime knots up to 8-crossings, we have not listed these since they can be easily identified.

\section{Explanation of appendices}

\subsection{Appendix \ref{appendix surgeries for MR1}}

Each figure shows a link and a coherent band surgery. The link is either (a) a $T(n,n)$ torus link or (b) the result of performing a coherent band surgery on two oppositely-oriented components of a $T(n,n)$ link to obtain an extra unlinked component. The band surgery shown in the diagram is that which will give the resulting 2-component link listed in the figure caption. These correspond to the cobordism $C_{L,a}$ described in the proof of Theorem \ref{mainresult1}.

\subsection{Appendix \ref{appendix surgeries knots to links}}

These diagrams show the coherent band surgeries required  for the knots listed in the introduction to satisfy Theorem \ref{mainresult1}.

\bibliographystyle{alpha}
\bibliography{test}

\begin{thebibliography}{{Ait}09}

\bibitem[{Ait}09]{AitNouh2009}
M.~{Ait Nouh}.
\newblock Genera and degrees of torus knots in {$CP^2$}.
\newblock {\em Journal of Knot Theory and Its Ramifications}, 2009.

\bibitem[FM66]{Fox1966}
R.~H. Fox and J.~Milnor.
\newblock Singularities of 2-spheres in 4-space and cobrdism of knots.
\newblock {\em Osaka Journal of Mathematics}, 1966.

\bibitem[Fox62a]{Fox1962a}
R.~H. Fox.
\newblock {\em Topology of 3-manifolds and related topics}, chapter A quick
  trip through knot theory, pages 120,167.
\newblock Prentice-Hall, 1962.

\bibitem[Fox62b]{Fox1962b}
R.~H. Fox.
\newblock {\em Topology of 3-manifolds and related topics}, chapter Some
  problems in knot theory, pages 168,176.
\newblock Prentice-Hall, 1962.

\bibitem[Gil81]{Gilmer1981}
P.~Gilmer.
\newblock Configurations of surfaces in 4-manifolds.
\newblock {\em Transactions of the American Mathematical Society}, 1981.

\bibitem[JK18]{Jabuka2018}
S.~Jabuka and T.~Kelly.
\newblock The nonorientable 4-genus for knots with 8 or 9 crossings.
\newblock {\em Algebraic \& Geometric Topology}, 2018.

\bibitem[KM94]{Kronheimer1994}
P.~Kronheimer and T.~Mrowka.
\newblock The genus of embedded surfaces in the projective plane.
\newblock {\em Mathematics Research Letters}, 1994.

\bibitem[Law92]{Lawson1992}
T.~Lawson.
\newblock Smooth embeddings of 2-spheres in 4-manifolds.
\newblock {\em Expositiones Mathematicae}, 1992.

\bibitem[LCa]{knotinfo}
C.~Livingston and J.~C. Cha.
\newblock {KnotInfo: Table of Knot Invariants}.
\newblock {https://www.indiana.edu/~knotinfo/}.
\newblock {Accessed: 2019-11-30}.

\bibitem[LCb]{linkinfo}
C.~Livingston and J.~C. Cha.
\newblock {LinkInfo: Table of Link Invariants}.
\newblock {http://www.indiana.edu/~linkinfo/}.
\newblock {Accessed: 2019-11-30}.

\bibitem[Lis07]{Lisca2007}
P.~Lisca.
\newblock Lens spaces, rational balls and the ribbon conjecture.
\newblock {\em Geometry \& Topology}, 2007.

\bibitem[LM17]{Lewark2017}
L.~Lewark and D.~McCoy.
\newblock {On Calculating the Slice Genera of 11- and 12-crossing Knots}.
\newblock {\em Experimental Mathematics}, 2017.

\bibitem[LN]{livnai}
C.~Livingston and S.~Naik.
\newblock {Introduction to Knot Concordance}.
\newblock (Work in Progress) {Accessed 2019 11-30}.

\bibitem[MV18]{Moore2018}
A.~Moore and M.~Vasquez.
\newblock A note on band surgery and the signature of a knot.
\newblock {arXiv:1806.02440}, 2018.

\bibitem[OS03]{Ozvath2003}
P.~Ozvath and Z.~Szabo.
\newblock {Knot Floer homology and the four-ball genus}.
\newblock {\em Geometry \& Topology}, 2003.

\bibitem[Rob65]{Robertello1965}
R.~Robertello.
\newblock An invariant of knot cobordism.
\newblock {\em Communications on Pure and Applied Mathematics}, 1965.

\bibitem[Sco05]{Scorpan2005}
A.~Scorpan.
\newblock {\em {The WIld World of 4-Manifolds}}.
\newblock American Mathematical Society, 2005.

\bibitem[Sei35]{Seifert1935}
H.~Seifert.
\newblock {\"Uber das Geschlect von Knoten}.
\newblock {\em Mathematische Annalen}, 1935.
\newblock In German.

\bibitem[Suz69]{Suzuki1969}
S.~Suzuki.
\newblock Local knots of 2-spheres in 4-manifolds.
\newblock In {\em Proceedings of the Japan Academy}, 1969.

\bibitem[Vir70]{Viro1970}
O.~Y. Viro.
\newblock Link types in codimension-2 with boundary.
\newblock {\em Uspekhi Matematicheskikh Nauk}, 1970.
\newblock In Russian.

\bibitem[Wil08]{Williams2008}
L.~Williams.
\newblock {Obstructing sliceness in a family of Montesinos knots}.
\newblock {arXiv:0809.1247}, 2008.

\bibitem[Yas91]{Yasuhara1991}
A.~Yasuhara.
\newblock {(2,15)-torus Knot is not Slice in {$CP^2$}}.
\newblock In {\em Proceedings of the Japan Academy}, 1991.

\bibitem[Yas92]{Yasuhara1992}
A.~Yasuhara.
\newblock On slice knots in the complex projective plane.
\newblock {\em Revista de la Mathematica}, 1992.

\bibitem[Yas96]{Yasuhara1996}
A.~Yasuhara.
\newblock Connecting lemmas and representing homology classes of
  simply-connected 4-manfiolds.
\newblock {\em Tokyo Journal of Mathematics}, 1996.

\end{thebibliography}

%
%

\appendix
\captionsetup[figure]{labelformat=empty}

\pagebreak
\section{Coherent band surgeries required for Theorem \ref{mainresult1}}\label{appendix surgeries for MR1}

\begin{figure}[h!]
    \centering
    \begin{minipage}{0.33\textwidth}
        \centering
        \includegraphics[width=0.6\textwidth]{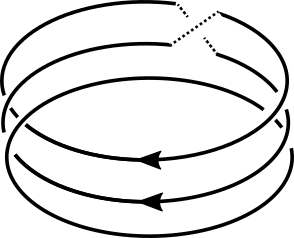} 
        \caption{L4a1$\{0\}\{1\}$ (Degree $ 1,3$)}
    \end{minipage}\hfill  
    \begin{minipage}{0.33\textwidth}
        \centering
        \includegraphics[width=0.6\textwidth]{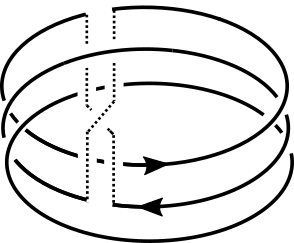} 
        \caption{$m$L7a4$\{0\}\{1\}$ (Degree $ 1$)}
    \end{minipage}
    \begin{minipage}{0.33\textwidth}
        \centering
        \includegraphics[width=0.6\textwidth]{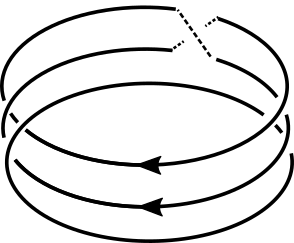} 
        \caption{$m$L7n1$\{0\}\{1\}$ (Degree $ 1, 3$)}
    \end{minipage}
    \vspace{0.75cm}
\end{figure}

\begin{figure}[h!]
    \centering
    \begin{minipage}{0.5\textwidth}
        \centering
        \includegraphics[width=0.9\textwidth]{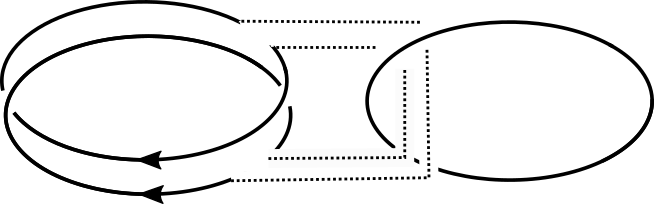} 
        \caption{$m$L5a1$\{0\}$ (Degree $2$)}
    \end{minipage}\hfill  
    \begin{minipage}{0.5\textwidth}
        \centering
        \includegraphics[width=0.9\textwidth]{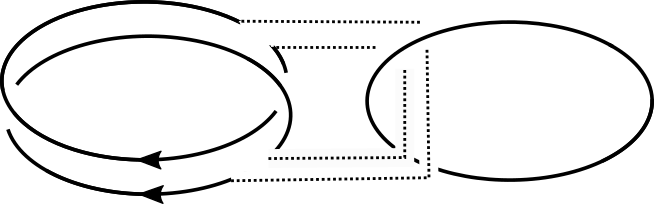} 
        \caption{L5a1$\{0\}$ (Degree $0$)}
    \end{minipage}
    \vspace{0.75cm}
      
\end{figure}

\begin{figure}[h!]
    \centering
    \begin{minipage}{0.5\textwidth}
        \centering
        \includegraphics[width=0.9\textwidth]{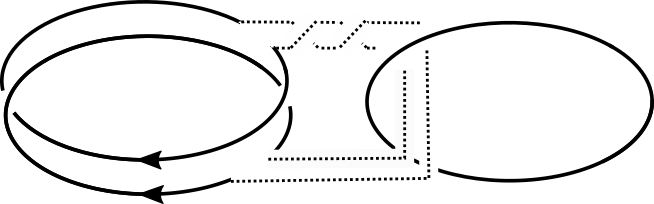} 
        \caption{$m$L7n2$\{0\}\{1\}$ (Degree $2$)}		
    \end{minipage}\hfill  
    \begin{minipage}{0.5\textwidth}
        \centering
        \includegraphics[width=0.9\textwidth]{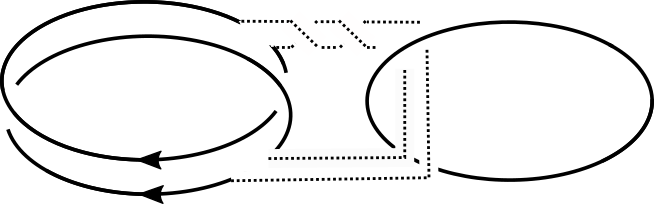} 
        \caption{L7n2$\{0\}\{1\}$ (Degree $0$)}
    \end{minipage}
    \vspace{0.75cm}  

\end{figure}

\begin{figure}[h!]
    \centering
    \begin{minipage}{0.66\textwidth}
        \centering
        \includegraphics[width=0.7\textwidth]{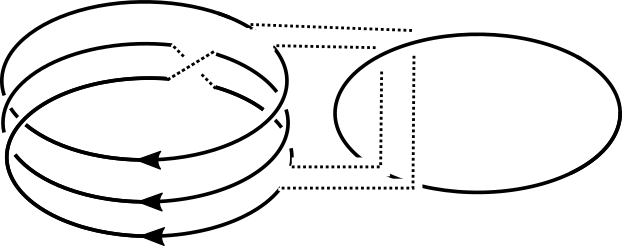} 
        \caption{L7a3$\{0\}\{1\}$ (Degree $ 3$)}
    \end{minipage}\hfill  
    \begin{minipage}{0.33\textwidth}
        \centering
        \includegraphics[width=0.6\textwidth]{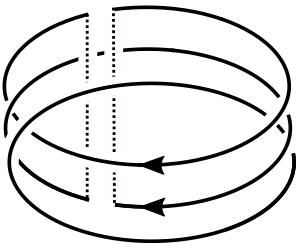} 
        \caption{$3_1\#m$L4a1$\{0\}\{1\}$ (Degree $ 1, 3$)}
    \end{minipage}
    \vspace{0.75cm}  

\end{figure}

\pagebreak
\section{Coherent band surgeries from specific knots to links in Theorem \ref{mainresult1}}\label{appendix surgeries knots to links}

\begin{figure}[h!]
    \centering
    \begin{minipage}{0.33\textwidth}
        \centering
        \includegraphics[width=0.75\textwidth]{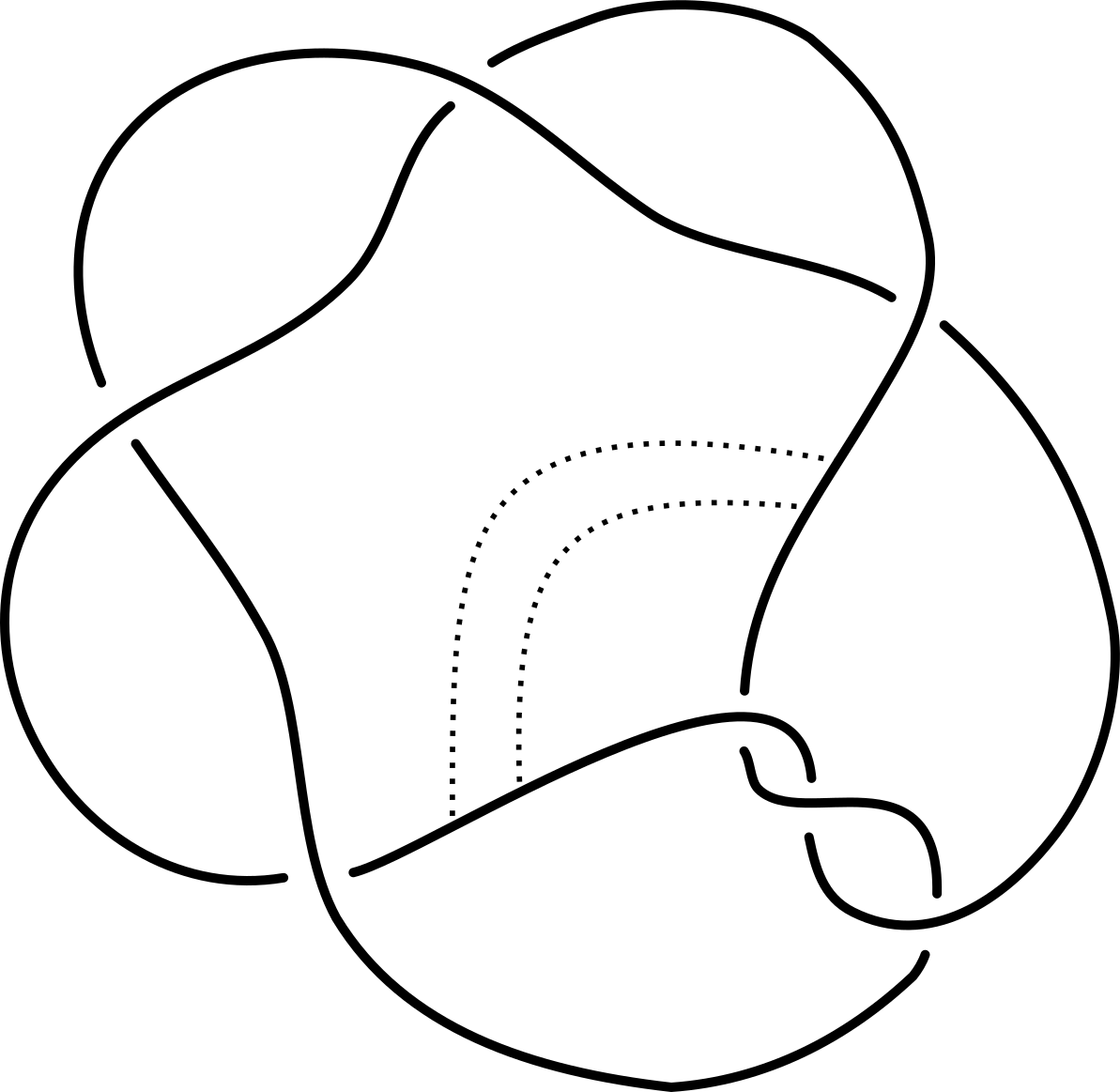} 
        \caption{$7_3\longrightarrow$ L4a1$\{1\}$}
    \end{minipage}\hfill  
    \begin{minipage}{0.33\textwidth}
        \centering
        \includegraphics[width=0.6\textwidth]{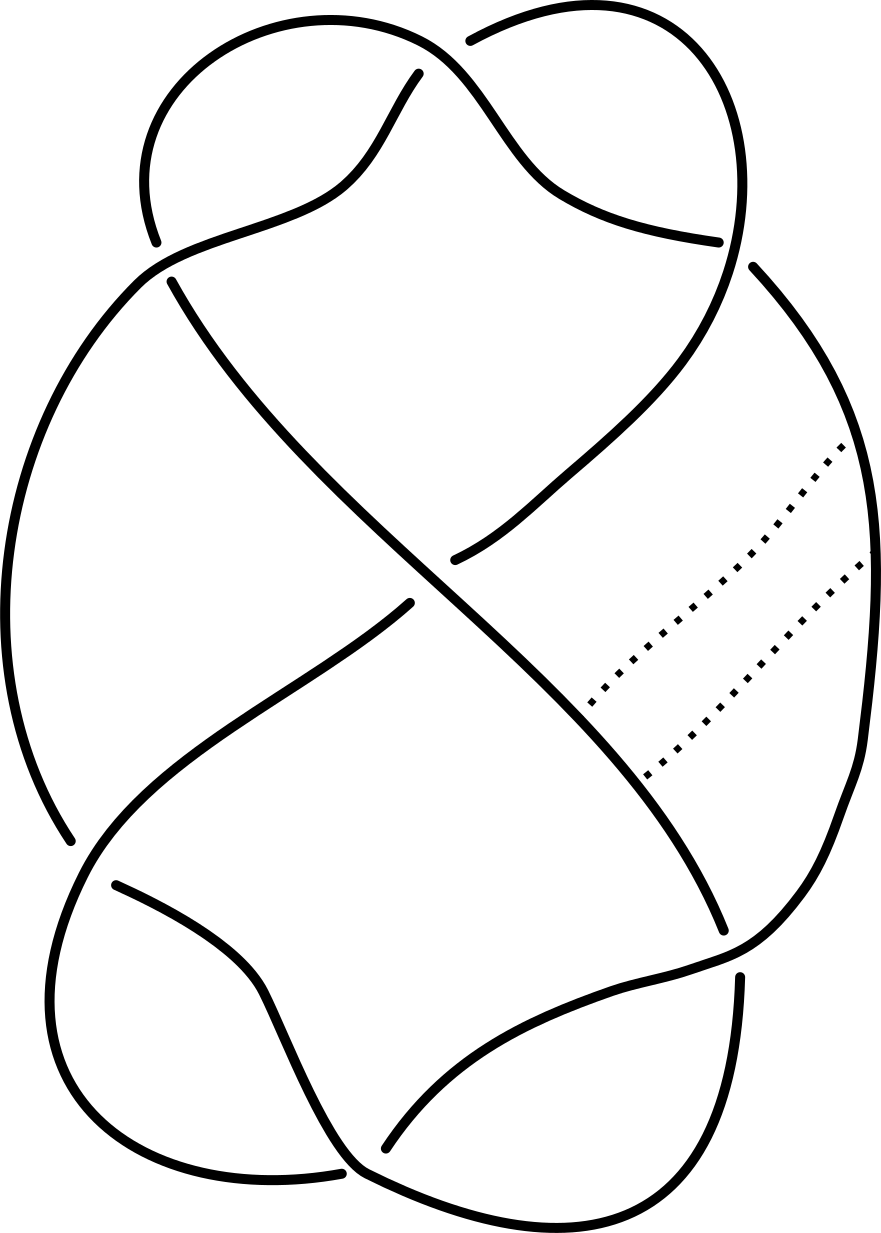} 
        \caption{$m7_4\longrightarrow$ L4a1$\{1\}$}
    \end{minipage}
    \begin{minipage}{0.33\textwidth}
        \centering
        \includegraphics[width=0.9\textwidth]{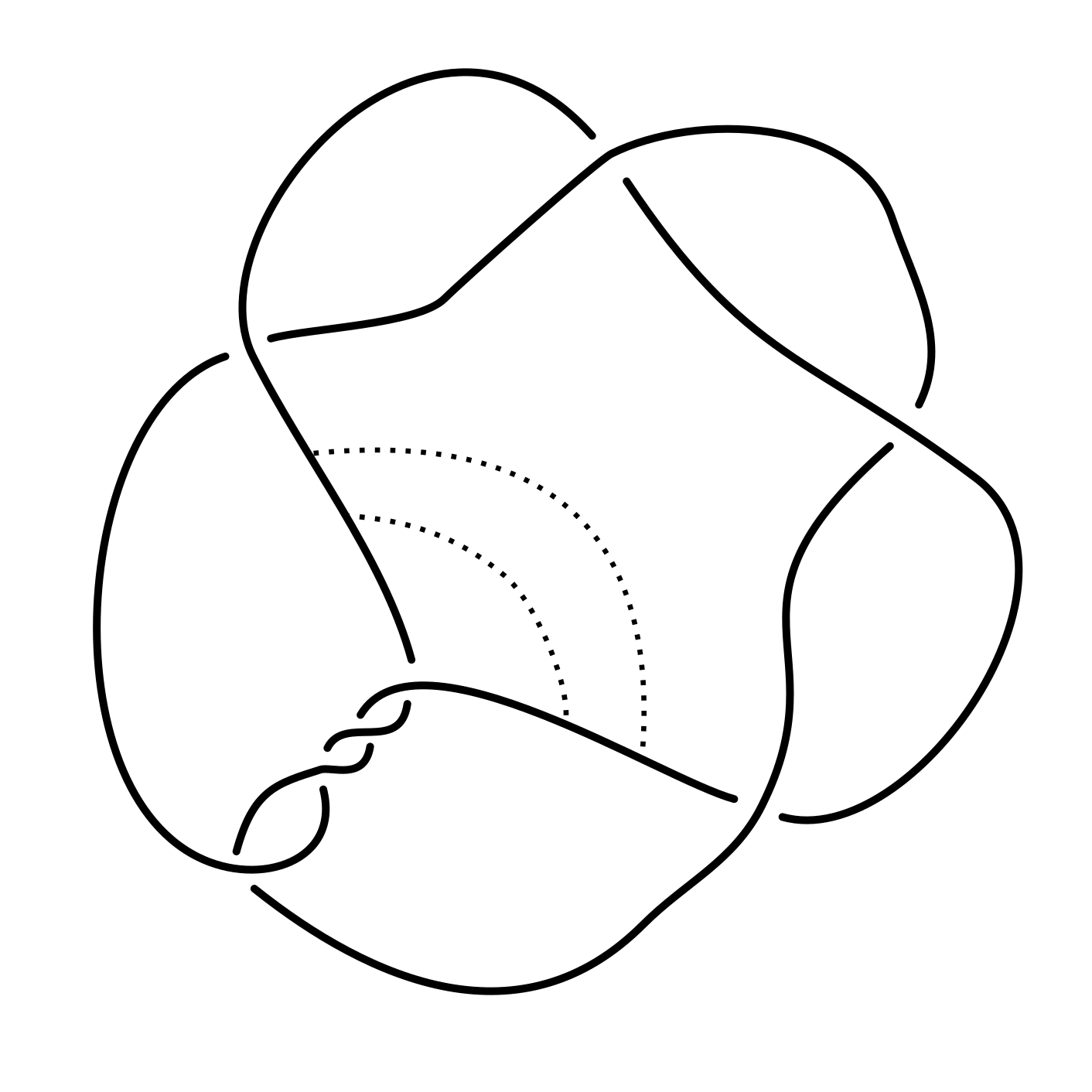} 
        \caption{$m8_3\longrightarrow$ L4a1$\{0\}$}
    \end{minipage}
    \vspace{0.75cm}

    \centering
    \begin{minipage}{0.33\textwidth}
        \centering
        \includegraphics[width=0.8\textwidth]{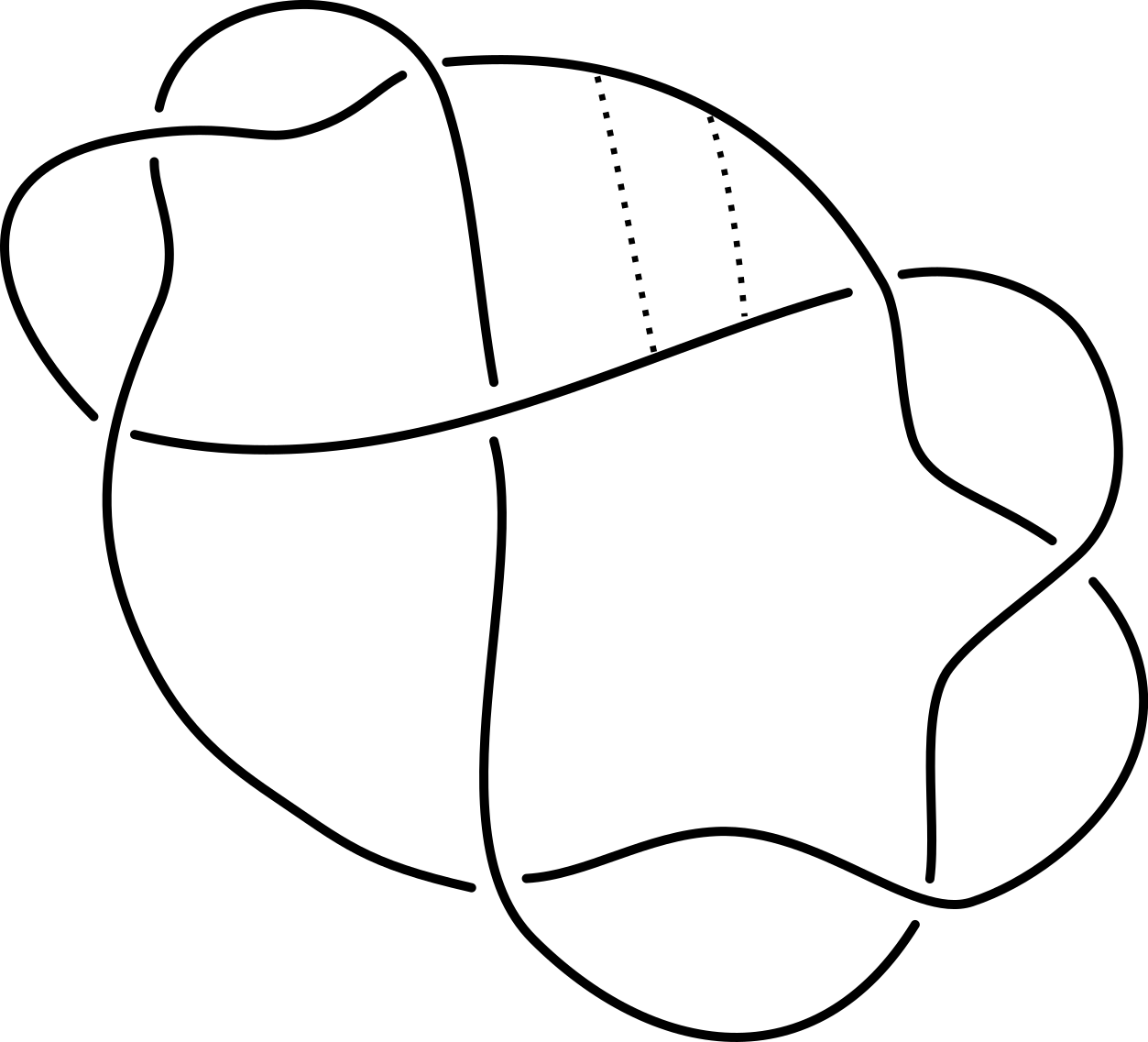} 
        \caption{$m8_4\longrightarrow$ L4a1$\{1\}$}
    \end{minipage}\hfill  
    \begin{minipage}{0.33\textwidth}
        \centering
        \includegraphics[width=0.95\textwidth]{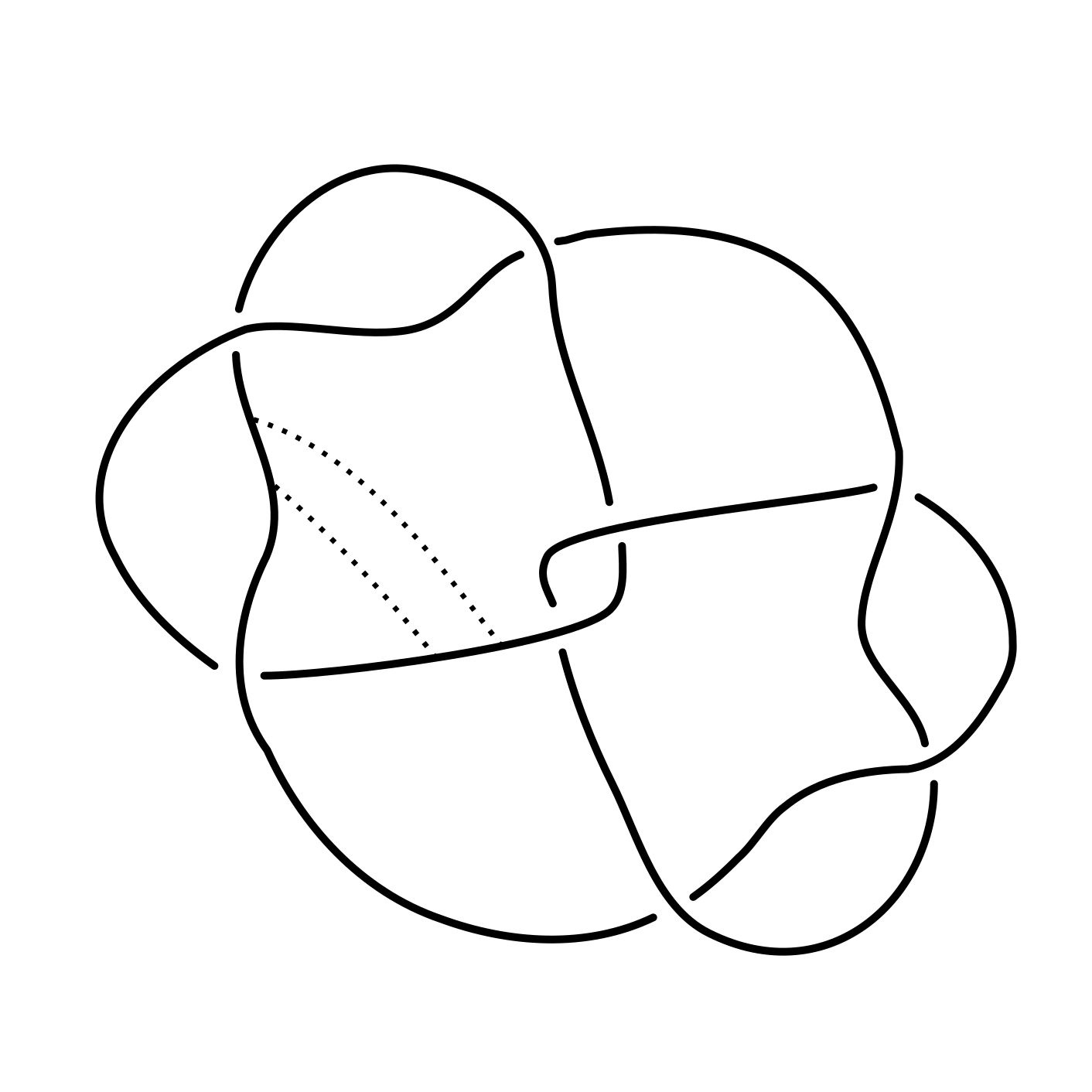} 
        \caption{$8_5\longrightarrow$ L4a1$\{1\}$}
    \end{minipage}
    \begin{minipage}{0.33\textwidth}
        \centering
        \includegraphics[width=0.75\textwidth]{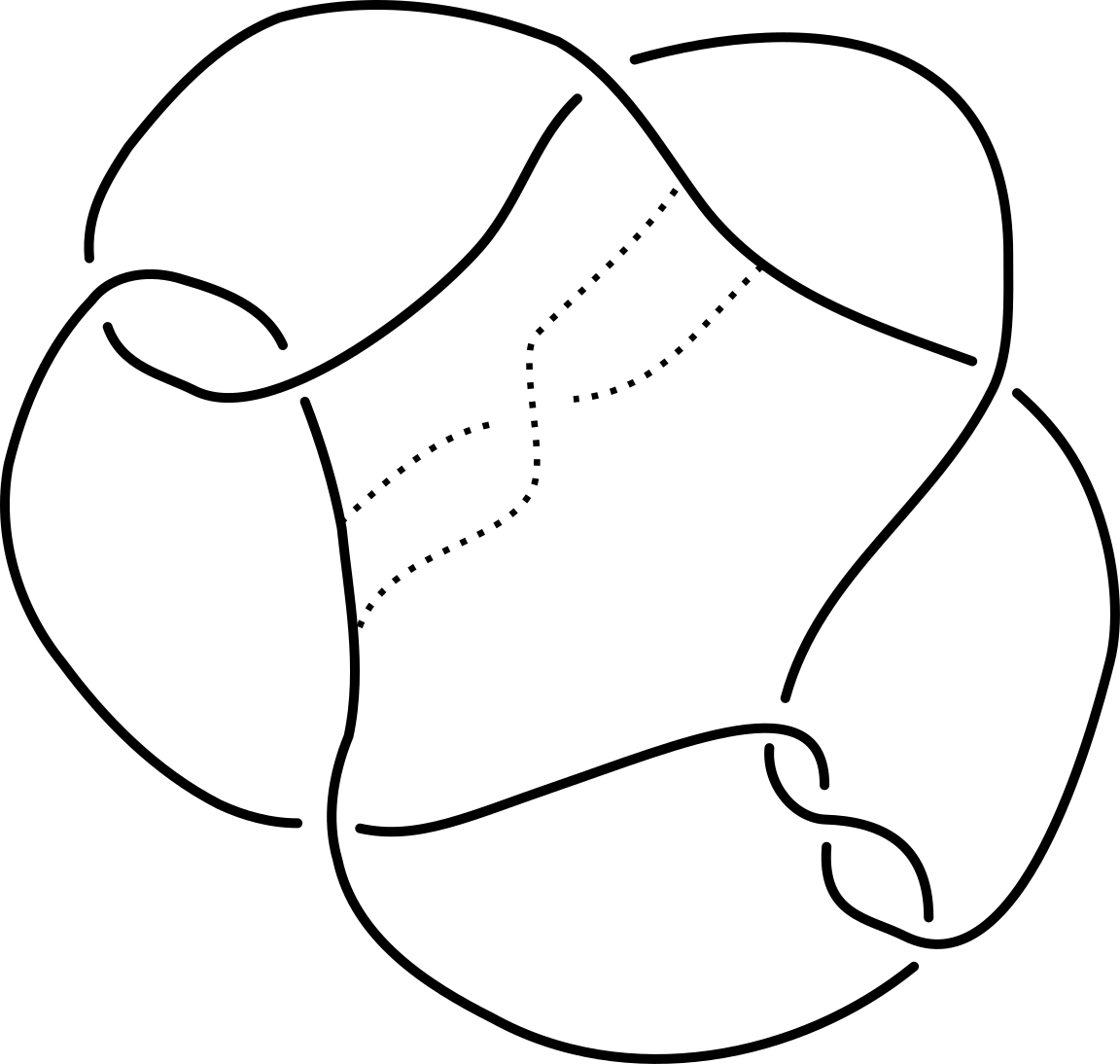} 
        \caption{$m8_6\longrightarrow$ $m$L2a1$\{1\}$}
    \end{minipage}
    \vspace{0.75cm}

    \centering
    \begin{minipage}{0.33\textwidth}
        \centering
        \includegraphics[width=0.7\textwidth]{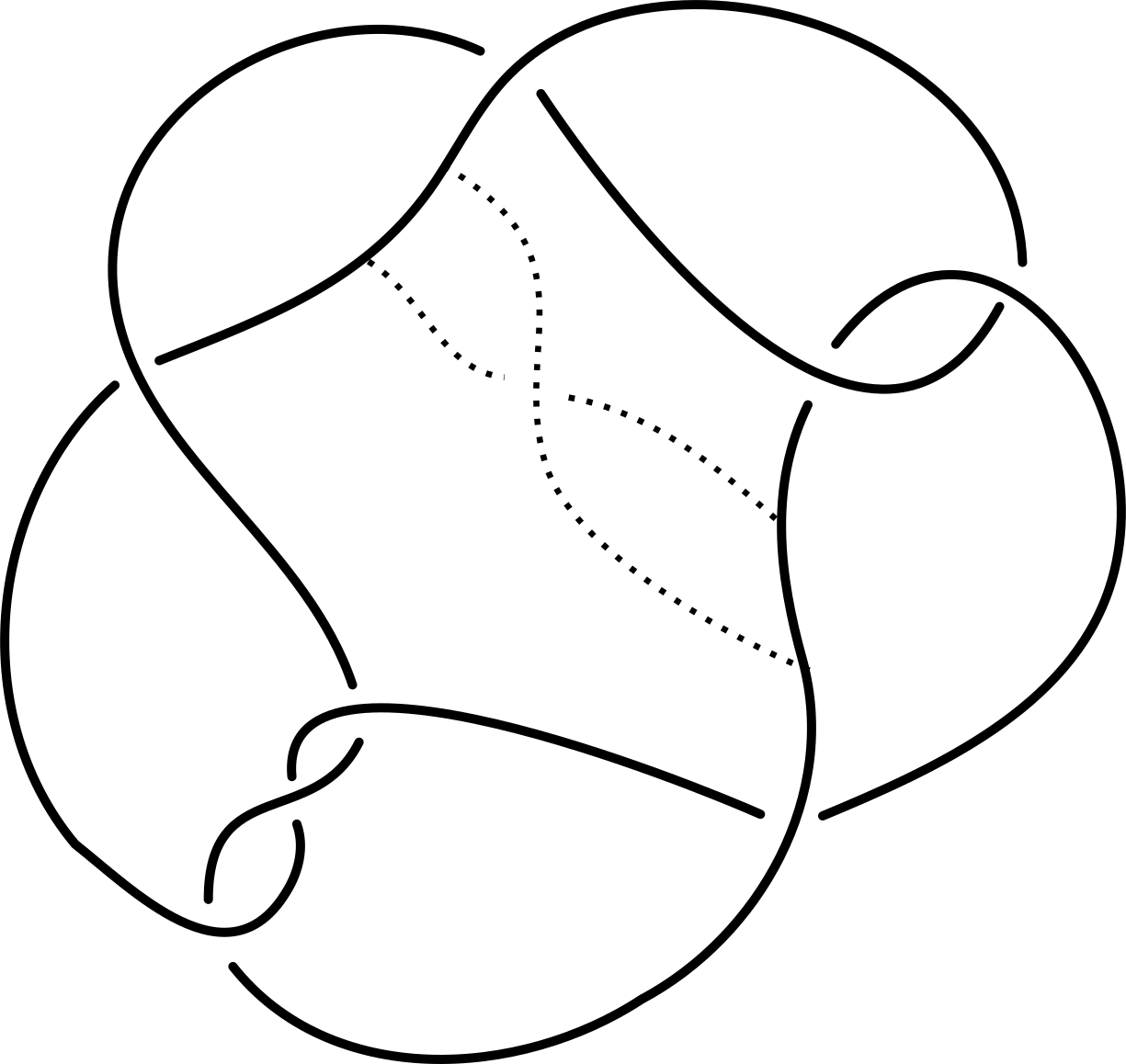} 
        \caption{$8_6\longrightarrow$ L2a1$\{1\}$}
    \end{minipage}\hfill  
    \begin{minipage}{0.33\textwidth}
        \centering
        \includegraphics[width=0.7\textwidth]{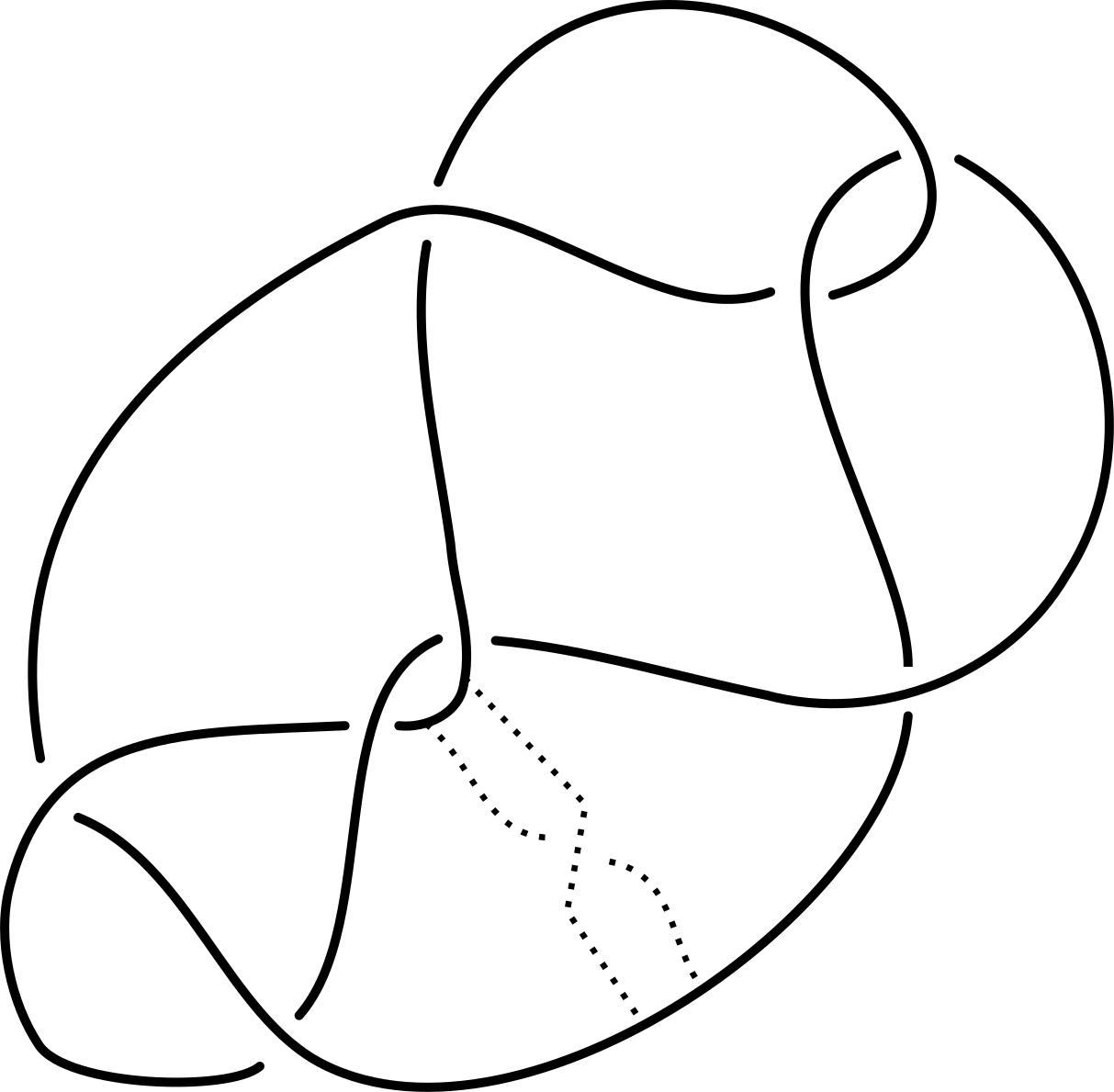} 
        \caption{$m8_{12}\longrightarrow$ L5a1$\{0\}$}
    \end{minipage}\hfill 
    \begin{minipage}{0.33\textwidth}
        \centering
        \includegraphics[width=0.75\textwidth]{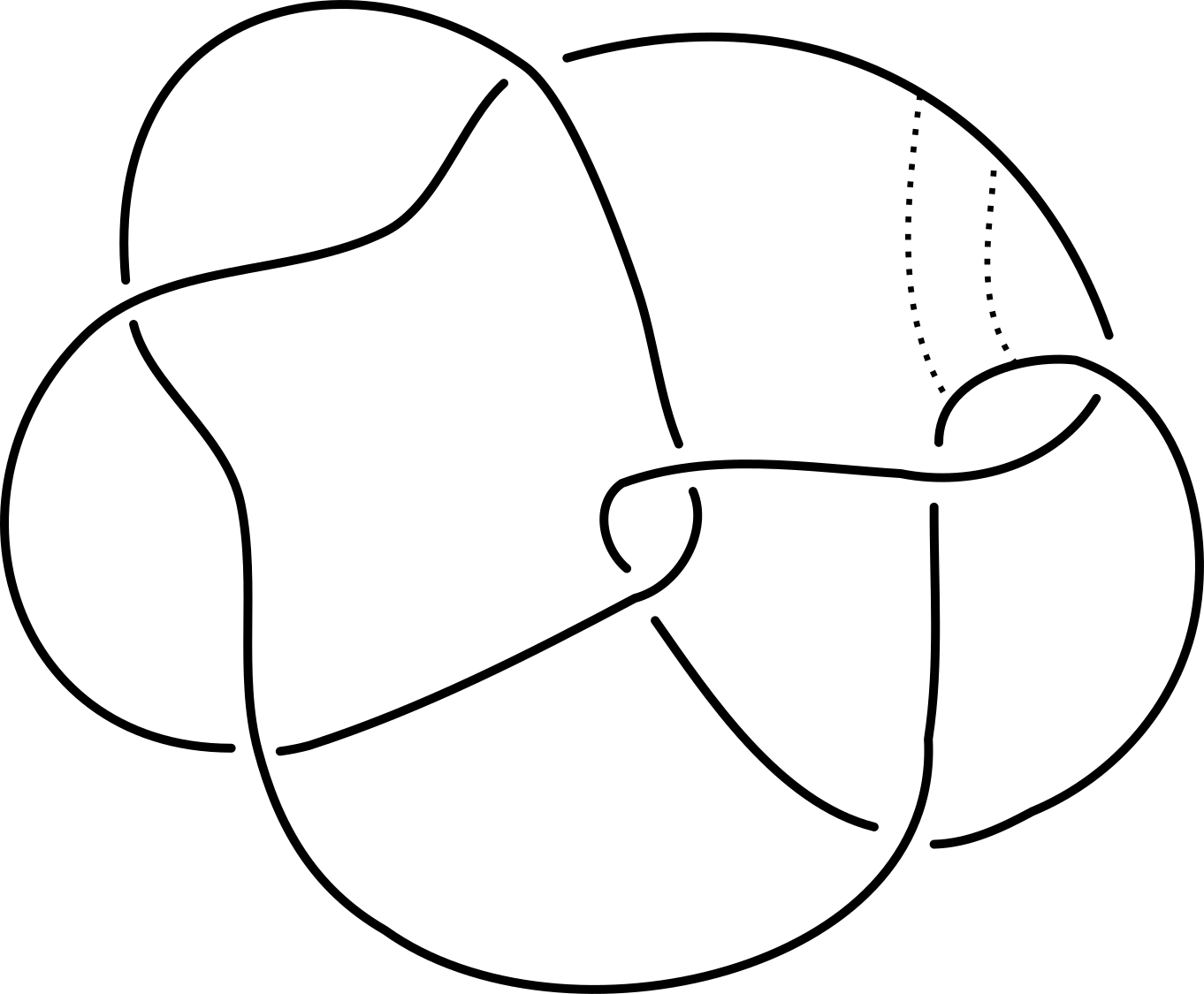} 
        \caption{$8_{19}\longrightarrow$ $m$L7n1$\{0\}$}
    \end{minipage}
    \vspace{0.75cm}
\end{figure}

\begin{figure}[h!]
    \centering
    \begin{minipage}{0.33\textwidth}
        \centering
        \includegraphics[width=0.7\textwidth]{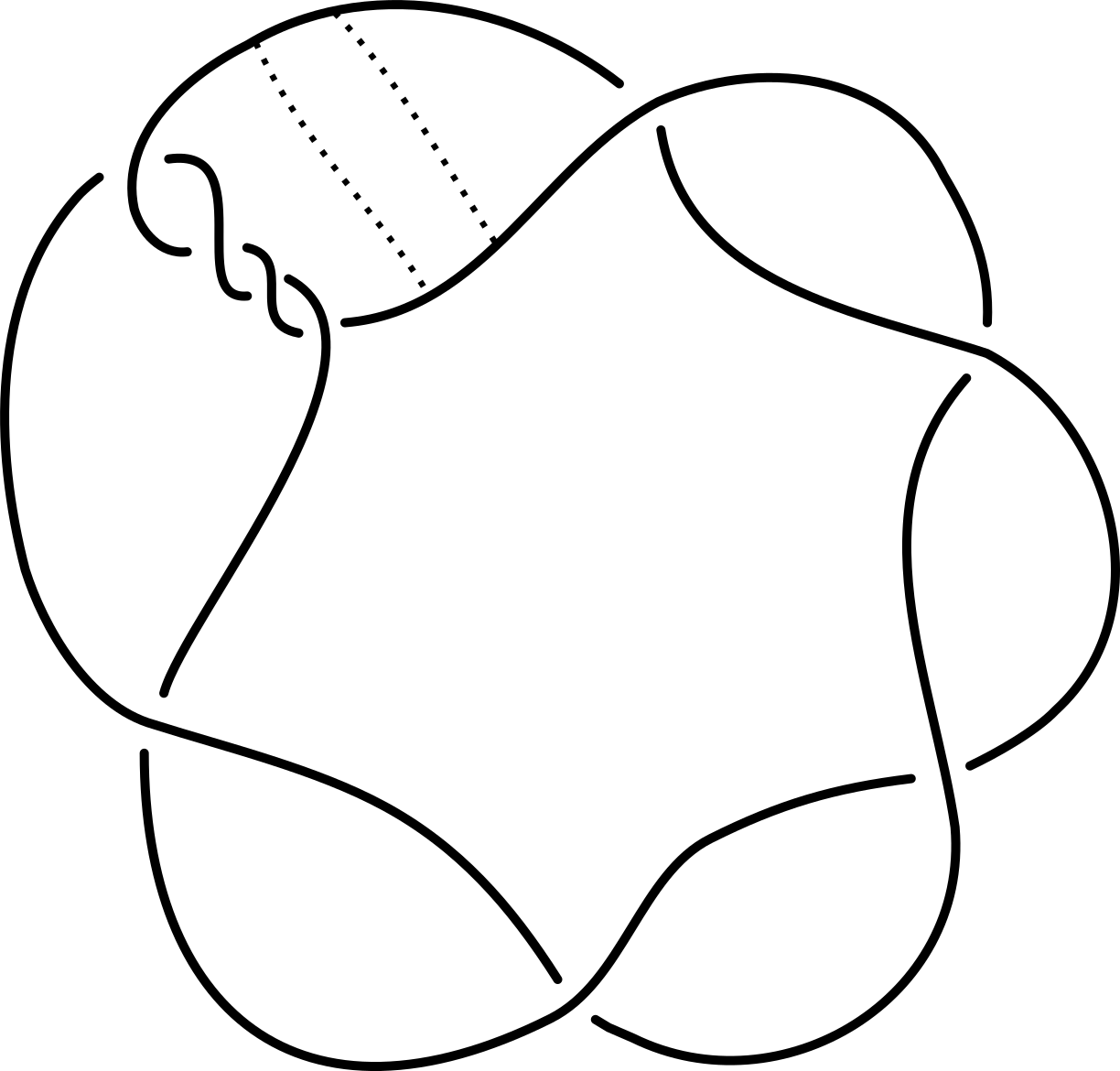} 
        \caption{$9_4\longrightarrow$ L4a1$\{1\}$}
    \end{minipage}\hfill
    \begin{minipage}{0.33\textwidth}
        \centering
        \includegraphics[width=0.6\textwidth]{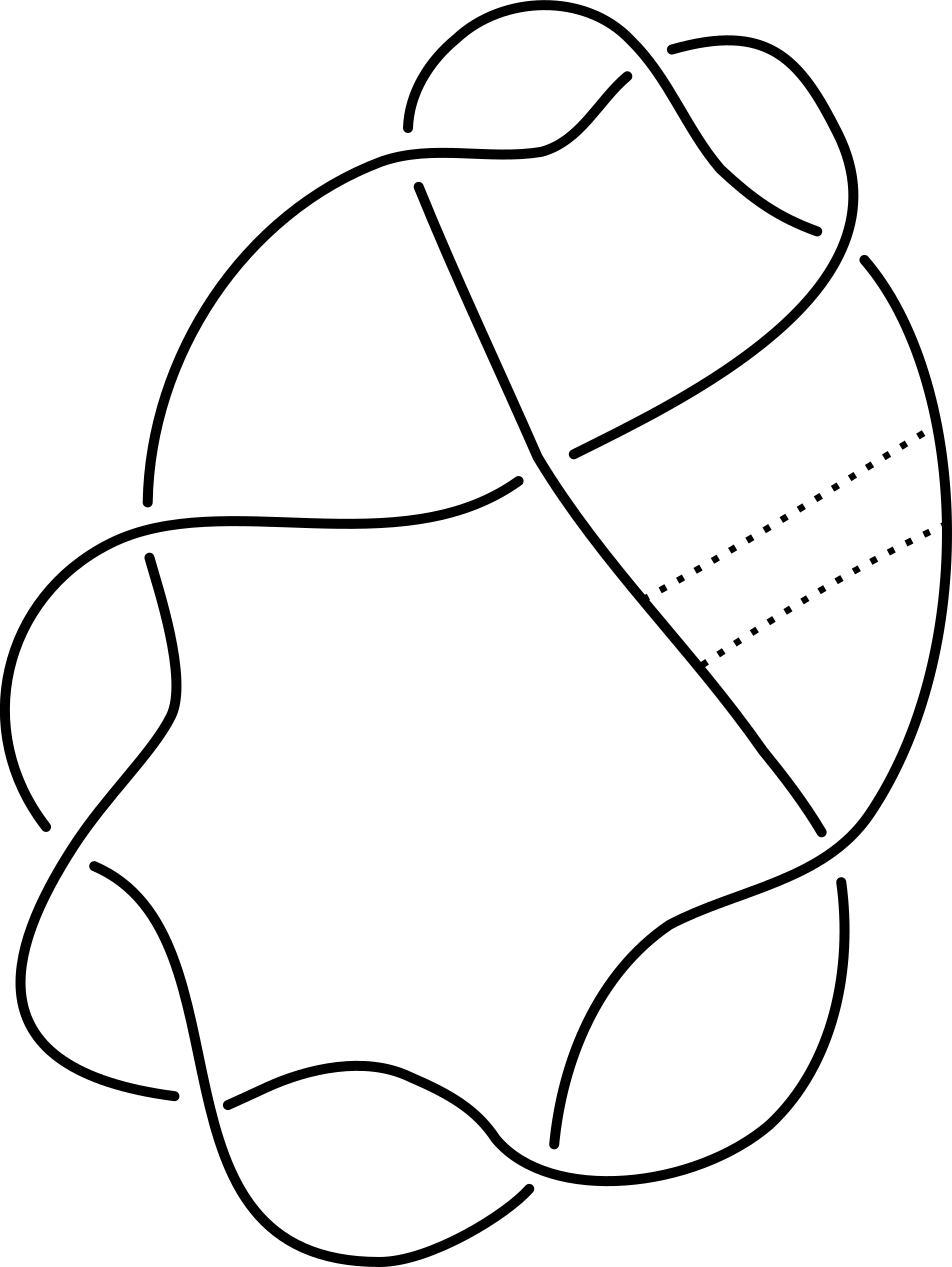} 
        \caption{$m9_5\longrightarrow$ L4a1$\{0\}$}
    \end{minipage}\hfill
    \begin{minipage}{0.33\textwidth}
        \centering
         \includegraphics[width=0.5\textwidth]{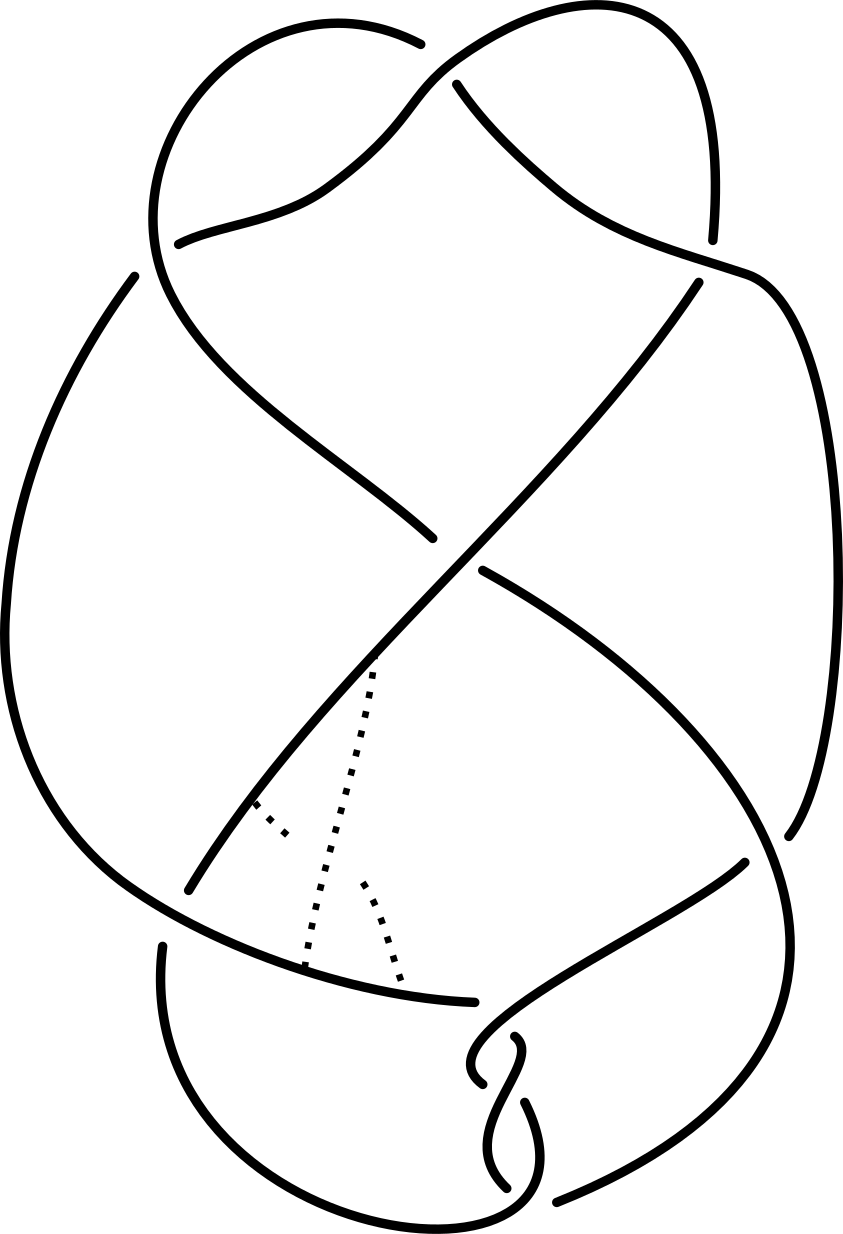} 
        \caption{$9_{13}\longrightarrow$ $3_1\#$ $m$L4a1$\{0\}$}
    \end{minipage}\hfill
    \vspace{0.75cm}
    
\end{figure}

\begin{figure}[h!]
    \centering
    \begin{minipage}{0.33\textwidth}
        \centering
        \includegraphics[width=0.7\textwidth]{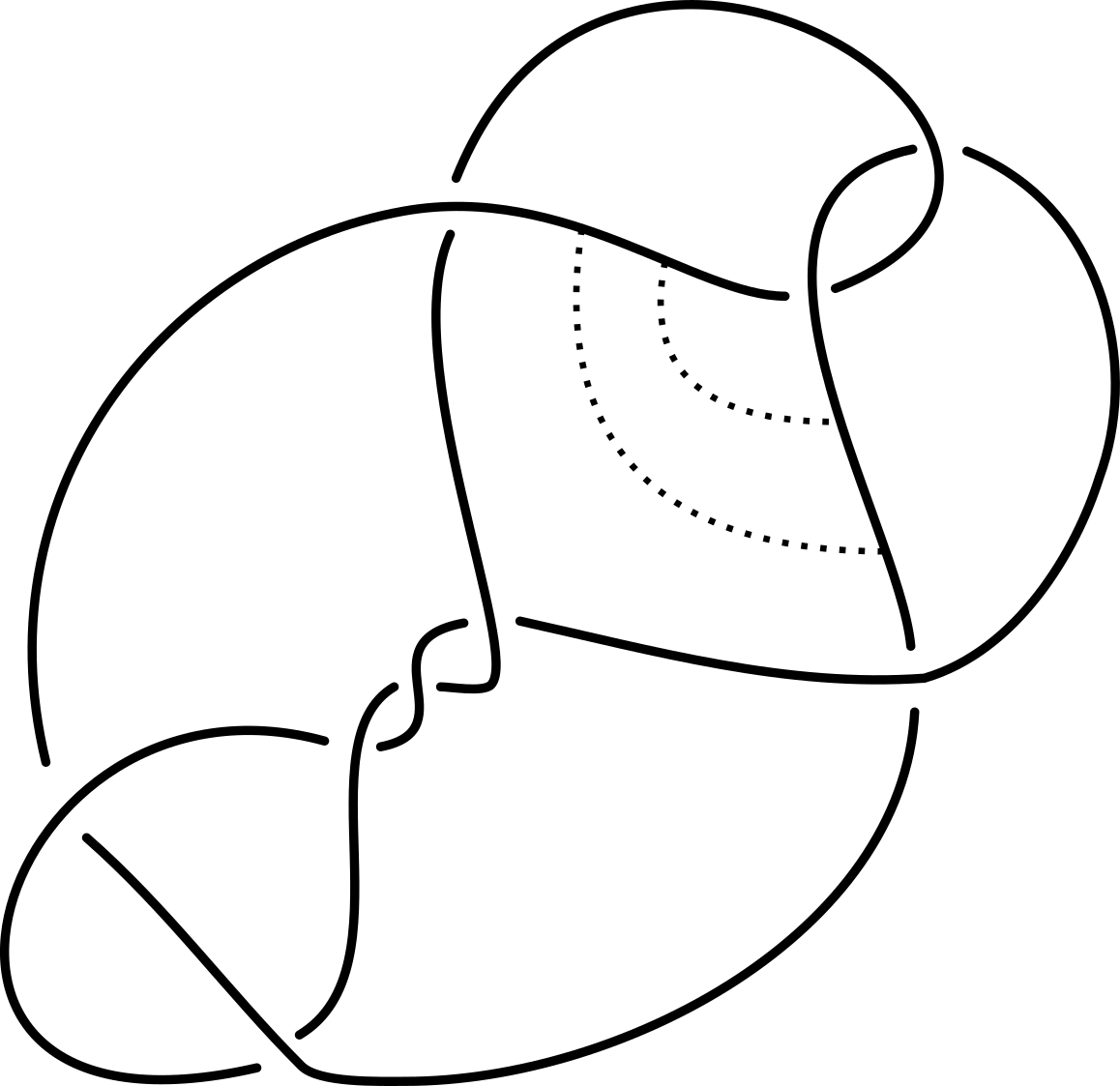} 
        \caption{$9_{15}\longrightarrow$ L7a4$\{0\}$}
    \end{minipage}\hfill
    \begin{minipage}{0.33\textwidth}
        \centering
        \includegraphics[width=0.6\textwidth]{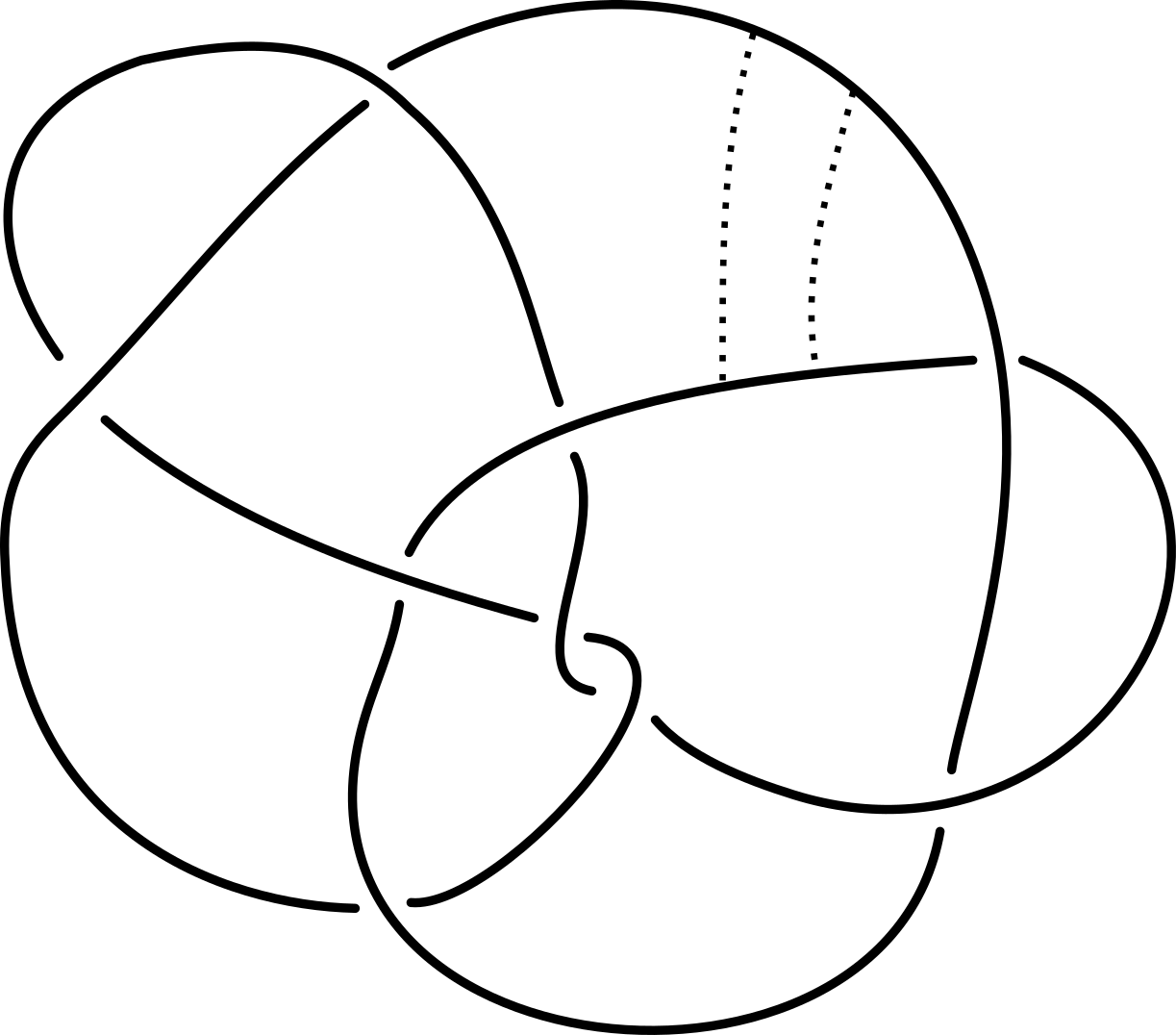} 
        \caption{$m9_{29}\longrightarrow$ L7a3$\{0\}$}
    \end{minipage}\hfill
    \begin{minipage}{0.33\textwidth}
        \centering
        \includegraphics[width=0.6\textwidth]{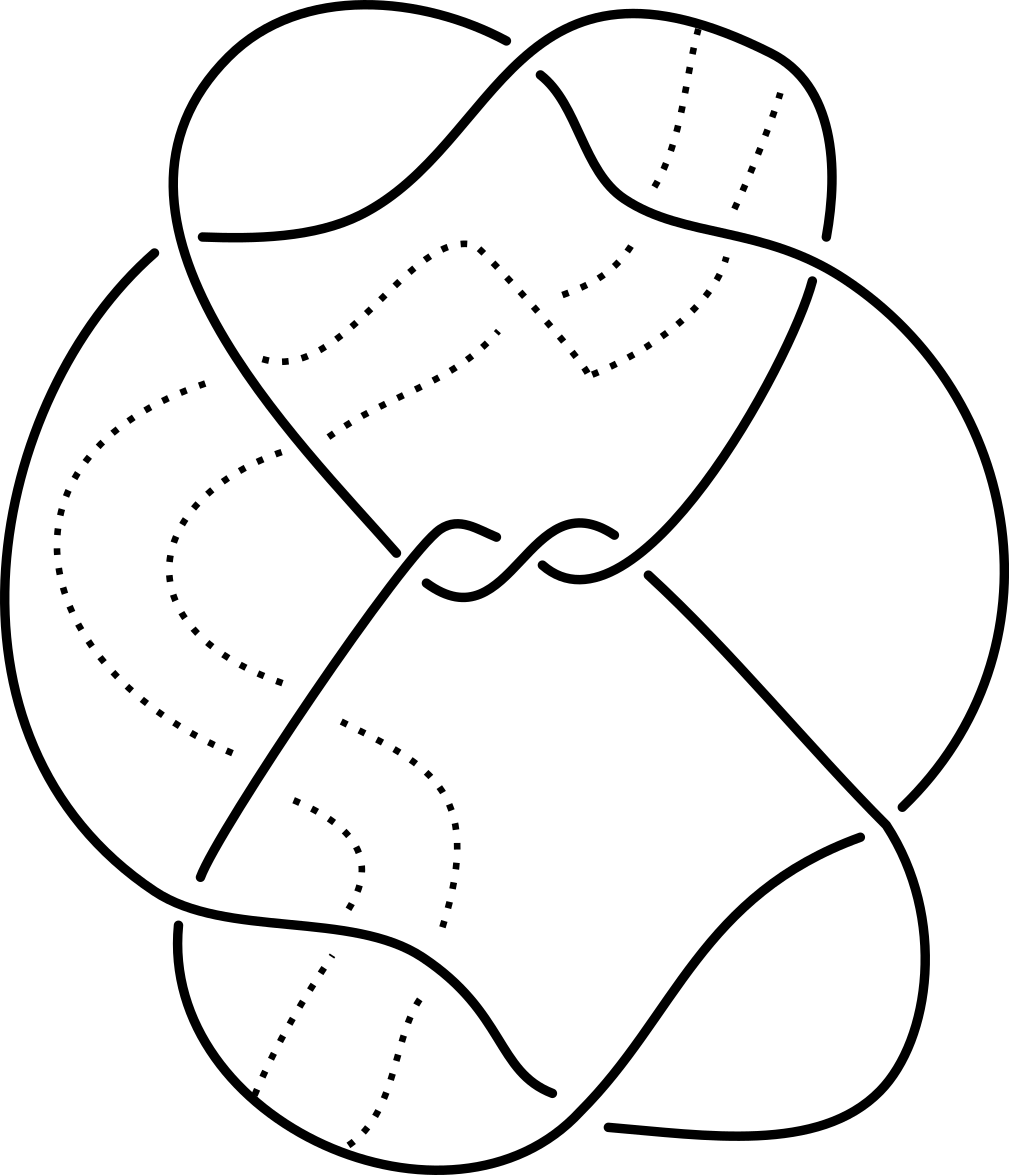} 
        \caption{$9_{35}\longrightarrow$ $3_1\#$ $m$L4a1$\{0\}$}
    \end{minipage}\hfill
    \vspace{0.75cm}
    
\end{figure}

\begin{figure}[h!]
    \centering
    \begin{minipage}{0.33\textwidth}
        \centering
        \includegraphics[width=0.7\textwidth]{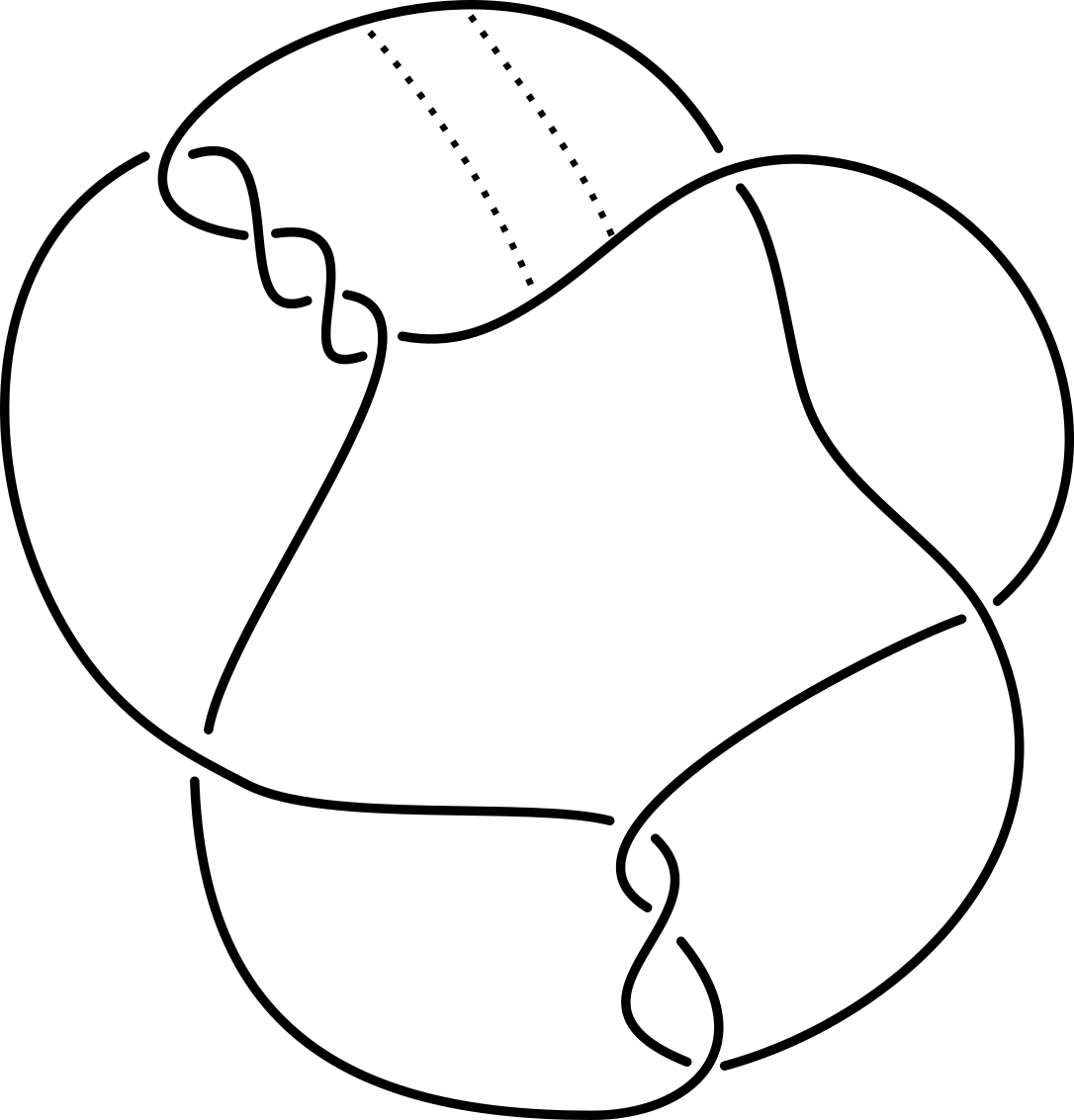} 
        \caption{$10_{11}\longrightarrow$ $3_1\#$ $m$L4a1$\{1\}$}
    \end{minipage}\hfill
    \begin{minipage}{0.33\textwidth}
        \centering
        \includegraphics[width=0.8\textwidth]{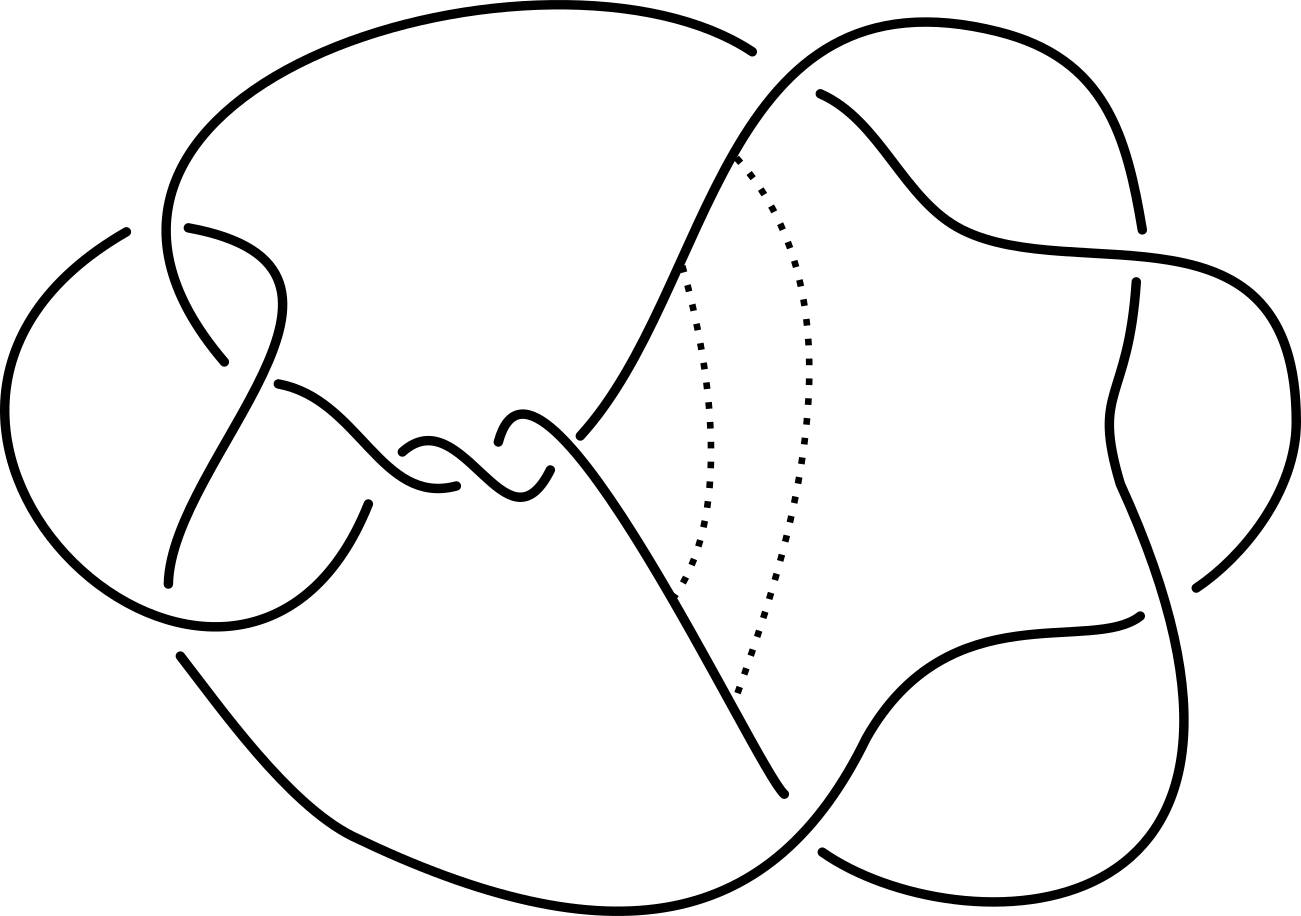} 
        \caption{$10_{12}\longrightarrow$ $3_1\#$ $m$L4a1$\{1\}$}
    \end{minipage}\hfill
    \begin{minipage}{0.33\textwidth}
        \centering
        \includegraphics[width=0.65\textwidth]{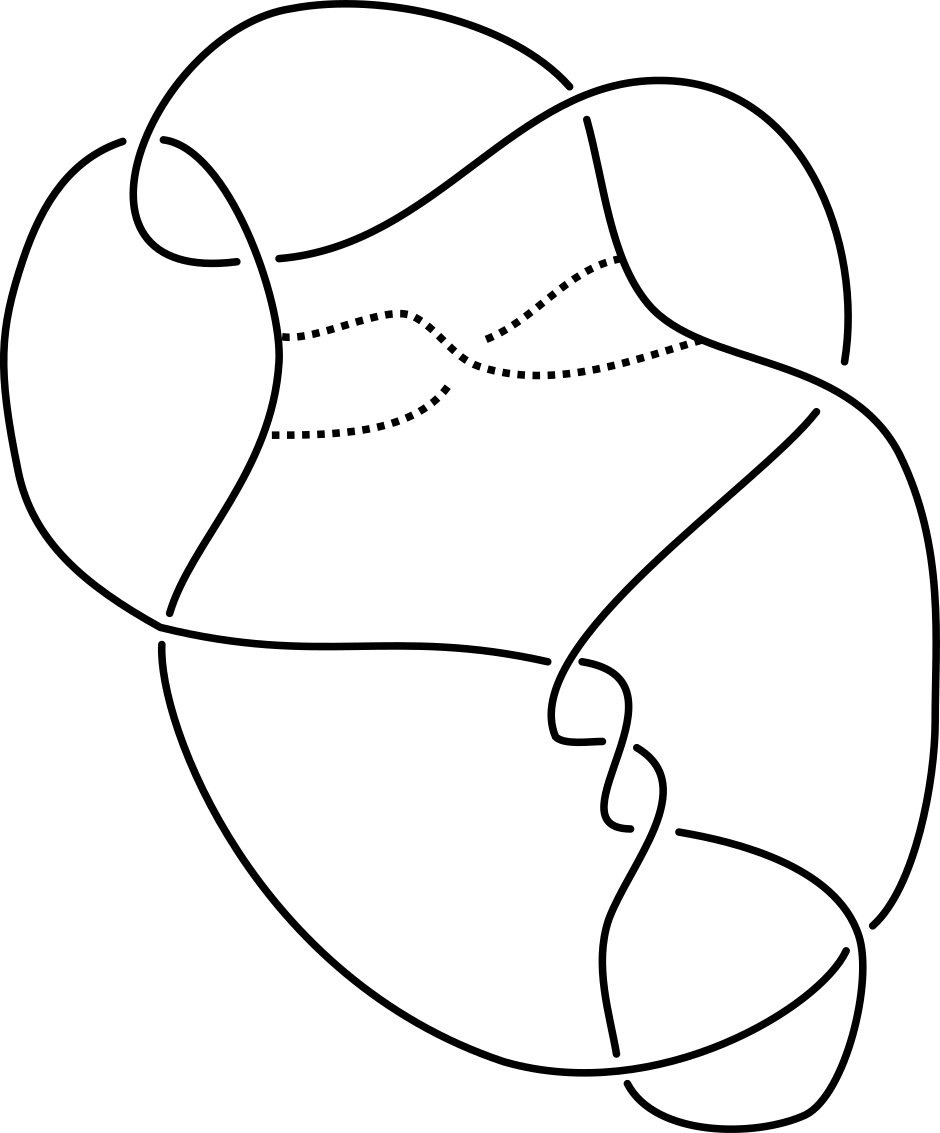}
        \caption{$10_{37}\longrightarrow$ $m$L5a1$\{0\}$}
    \end{minipage}\hfill
    \vspace{0.75cm}
    
\end{figure}

\end{document}